\documentclass[english,reqno]{amsart}

\usepackage{amsmath, appendix, ulem, mathtools,mcode}
\usepackage[nobysame]{amsrefs}
\usepackage{amssymb, color}
\usepackage[margin=1in]{geometry}

\usepackage{mathrsfs}
\usepackage{graphicx}
\usepackage{subfig}
\usepackage{float}
\usepackage{epsf}
\usepackage{hyperref}
\usepackage{titletoc}

\numberwithin{equation}{section}

\newtheorem{lem}{Lemma}[section]
\newtheorem{thm}{Theorem}[section]
\newtheorem{proposition}[thm]{Proposition}

\theoremstyle{remark}
\newtheorem{rmk}{Remark}[section]
\newtheorem{definition}[thm]{Definition}

\renewcommand{\tilde}{\widetilde}

\newcommand{\nn}{\nonumber}
\newcommand{\R}{{\mathbb R}}

\newcommand{\dt}{ \, {\rm d} t}
\newcommand{\dx}{ \, {\rm d} x}

\newcommand{\ds}{\, {\rm d} s}

\newcommand{\Eps}{\varepsilon}

\newtheorem{assum}{Assumption}

\newcommand{\EE}{\mathbb{E}}
\newcommand{\RR}{\mathbb{R}}

\newcommand{\PP}{\mathbb{P}}

\newcommand{\bfR}{\mathbf{R}}
\newcommand{\bfP}{{\pi}_{_{\partial D}}} 

\newcommand{\munuX}{\mu_t^{\nu,X}}

\newcommand{\rd}{\mathrm{d}}

\newcommand{\la}{\langle}
\newcommand{\ra}{\rangle}

\newcommand{\norm}[1]{\left\lVert#1 \, \right\rVert}

\newcommand{\viint}[2]{\left\langle#1, \, #2 \,\right\rangle}
\newcommand{\vpran}[1]{\left(#1\right)}

\newcommand{\npart}[2]{\tilde{X}_{\tau_{#1}}^{#2}}
\newcommand{\npartd}[2]{\tilde{X}_{\tau_{#1}}^{\nu, #2}}
\newcommand{\npartdw}[1]{\tilde{X}_{\tau_{#1}}^{\nu}}
\newcommand{\nempdif}[1]{\mu^{\nu,\tilde{X}}_{\tau_{#1}}}
\newcommand{\nemp}[1]{\mu^{\tilde{X}}_{\tau_{#1}}}
\newcommand{\klip}{K_2}
\newcommand{\knlip}{K_{3/2}}
\newcommand{\disk}{D_C}
\newcommand{\hp}{D_H}
\newcommand{\avW}[1]{E^\nu(#1)}

\begin{document}
\title[Zero-diffusion limit]{Zero-diffusion limit for aggregation equations\\ over bounded domains}
\author{Razvan C. Fetecau}  
\address{Department of Mathematics, Simon Fraser University, 8888 University Dr., Burnaby, BC V5A 1S6, Canada.}
\email{van@math.sfu.ca}

\author{Hui Huang}
\address{Department of Mathematics, Simon Fraser University, 8888 University Dr., Burnaby, BC V5A 1S6, Canada.}
\email{hha101@sfu.ca}

\author{Daniel Messenger}
\address{Department of Mathematics, Simon Fraser University, 8888 University Dr., Burnaby, BC V5A 1S6, Canada.}
\email{dmesseng@sfu.ca}
 
\author{Weiran Sun}
\address{Department of Mathematics, Simon Fraser University, 8888 University Dr., Burnaby, BC V5A 1S6, Canada.}
\email{weirans@sfu.ca}
 
\maketitle
\begin{abstract}
We establish the zero-diffusion limit for both continuous and discrete aggregation models over convex and bounded domains. Compared with a similar zero-diffusion limit derived in \cite{zhang2017continuity}, our approach is different and relies on a coupling method connecting PDEs with their underlying SDEs. Moreover, our result relaxes the regularity assumptions on the interaction and external potentials and improves the convergence rate (in terms of the diffusion coefficient). The particular rate we derive is shown to be consistent with numerical computations. 
\end{abstract}
\section{Introduction}
In this paper we establish the zero-diffusion limit ($\nu\rightarrow 0$) of weak measure-valued solutions to the following aggregation-diffusion equation:
\begin{align}\label{PDEnu}
\begin{cases}
\partial_t\mu^\nu_t =\nu\triangle \mu^\nu_t+\nabla \cdot[\mu^\nu_t (\nabla K\ast \mu^\nu_t+\nabla V)] \,, 
\quad &  t>0\,,\\[2pt]
\mu^\nu_t|_{t=0}=\mu_0\,, \\[2pt]
\langle \nu\nabla\mu^\nu_t+\mu^\nu_t (\nabla K\ast \mu^\nu_t+\nabla V),n\rangle=0 \,,
\quad &\mbox{on } \partial D \,,
\end{cases}
\end{align}
where $\la,\ra$ denotes the dot product in $\RR^d$ $(d\geq 2)$, $K$ is an interaction potential, $V$ is an external potential,  $\mu_t^\nu$ is a probability measure  and the equation is set on $D\subset \RR^d$, a closed, bounded convex domain with smooth boundary. Also, $n$ denotes the outward unit normal to $\partial D$, $\nu>0$ is the diffusion coefficient, and the convolution is defined~ as
\begin{equation}
\nabla K\ast \mu^\nu_t(x):=\int_D \nabla K(x-y) \, \rd \mu_t^\nu(y)\,.
\end{equation}

The objective of this paper is to show how solutions to \eqref{PDEnu} converge in the zero-diffusion limit to solutions of the plain aggregation model given by
\begin{align}\label{PDE}
\begin{cases}
\partial_t\mu_t =\nabla \cdot[\mu_tP_x(\nabla K\ast \mu_t+\nabla V)] \,, 
\quad  t>0\,,\\[2pt]
\mu_t|_{t=0}=\mu_0\,,
 \\[2pt]
\end{cases}
\end{align}
where the projection operator $P_x:\RR^d\rightarrow \RR^d$ is defined by
\begin{align}
P_x(v)=\begin{cases}
v\,,\quad&\mbox{if }x \in \mbox{int}(D) \mbox{ or if }x\in\partial D \mbox{ and } v\cdot n\leq 0\,,\\
\Pi_{\partial D}v\,,\quad &\mbox{otherwise.}
\end{cases}
\end{align}
Here $\mbox{int}(D)$ denotes the interior of the domain $D$ and $\Pi_{\partial D}v$ represents the projection of $v$ onto the tangent plane of the boundary at $x \in \partial D$ \cite{carrillo2014nonlocal,wu2015nonlocal}. This means that $P_x(v)=v$ everywhere in $D$, except at points $x \in \partial D$ where $v$ points outward the domain, in which case $v$ is projected onto the tangent plane of the boundary. Therefore, particles hitting the boundary with an outward pointing velocity do not exit the domain but move freely along it. 

Models \eqref{PDEnu} and \eqref{PDE} appear in many applications including material science \cite{holm2006formation}, robotics \cite{gazi2003stability}, granular flow \cite{carrillo2003kinetic}, 
biological swarming \cite{bertozzi2012aggregation}, and Ginzburg-Landau theory  \cite{weinan1994dynamics}. Mathematically, the well-posedness of~\eqref{PDEnu} and~\eqref{PDE} have been established both in the free space $\R^d$ and over bounded domains. In the case of $\R^d$, the global well-posedness of equation~\eqref{PDEnu} in the space of probability measures is established in \cite[Section 11.2]{ambrosio2008gradient} for certain classes of potentials, by using the theory of gradient flows. Meanwhile, depending on the regularity of the potential $K$, the plain aggregation equation~\eqref{PDE} in $\R^d$ is shown to have finite-time blow-up or global solutions \cite{bertozzi2009blow,bertozzi2007finite,carrillo2011global}. 

In our study the interest lies in domains with boundaries which are relevant in many realistic physical settings. Mathematically, the well-posedness of~\eqref{PDE} in the probability measure space has been established over bounded domains with a mild regularity 
condition \cite{carrillo2014nonlocal} and over Riemannian manifolds \cite{wu2015nonlocal}. Both results are obtained by using the gradient flow method. For equations with diffusion, the well-posedness has been studied in \cite{bedrossian2011local,bertozzi2009existence} for models with degenerate and/or linear diffusion. In a more general setting where the domain is time-dependent, well-posedness of equations of types~\eqref{PDEnu} and~\eqref{PDE} are both justified in~\cite{zhang2017continuity} again via the gradient flow method.

Our present work is motivated by the recent study in \cite{fetecau2017swarm}, where the authors identified a flaw of the plain aggregation equation \eqref{PDE} in domains with boundaries: its solutions can evolve into unstable equilibria. This is an surprising degeneracy of model \eqref{PDE} given that it has a gradient flow formulation. One approach to rectify this degeneracy is to regularize equation \eqref{PDE} by adding a small diffusion term (as in model \eqref{PDEnu}), in other words, by including a small Brownian noise, as most realistic setups of the aggregation model would inherently have. An immediate question is to investigate how close the regularized system is to the original plain aggregation model. Several works have been done in such direction: 
In free space, it has been shown in \cite{evers2016metastable}  that the linear diffusion can remove the degeneracy of multiple unstable equilibria and lead to a unique steady state. In ~\cite{fetecau2017swarming} the authors investigated the regularization of \eqref{PDE} by nonlinear diffusion. Most relevant to our work in this paper, in \cite{zhang2017continuity} the author studied the zero-diffusion limit of the linear diffusion model in moving domains with boundaries.

Different from the gradient flow method used in~\cite{zhang2017continuity}, our approach to prove the zero-diffusion limit is via a coupling method by considering self-consistent stochastic processes associated to \eqref{PDEnu} and \eqref{PDE}.
To be more precise, let $(\Omega,\mathcal{F},(\mathcal{F}_t)_{t\geq0},\PP)$ be a probability space endowed with a filtration $(\mathcal{F}_t)_{t\geq0}$. Suppose that $(B_t)_{t\geq 0}$ is a $d$-dimensional $\mathcal{F}_t$-Brownian motion that is independent of the (initial) random variable $Y_0$ with distribution $\mu_0$. The self-consistent stochastic process $(Y^{\nu}_t)_{t\geq0}$ underlying equation \eqref{PDEnu} satisfies 
\begin{equation}\label{SDEnu}
	\left\{\begin{array}{l}
		Y^{\nu}_t=Y_0-\int_{0}^t\left(\nabla K\ast\mu_s^\nu(Y^{\nu}_s) +\nabla V(Y^{\nu}_s)\right)\ds+\sqrt{2\nu}{B}_t-\tilde R_t^\nu\,, \quad t>0\,,\\
		\tilde R_t^\nu=\int_0^tn(Y^{\nu}_s) \, {\rm d}|\tilde R^\nu|_s\,,\quad |\tilde R^\nu|_t=\int_0^t\textbf{1}_{\partial D}(Y^{\nu}_s) \, {\rm d}|\tilde R^\nu|_s\,,
	\end{array}\right.
\end{equation}
where $\mu_t^\nu=Law(Y_t^\nu)$ is the probability distribution (or law) of $Y_t^\nu$. Here $\tilde R_t^\nu$ is a reflecting process associated to $Y^{\nu}_t$  with a bounded total variation. Moreover, $|\tilde R^\nu|_t$ is the total variation of $\tilde R_t^\nu$ on $[0,t]$, namely
\begin{equation}
	|\tilde R^\nu|_t=\sup\sum\limits_k|\tilde R_{t_k}^\nu-\tilde R_{t_{k-1}}^\nu|\,,
\end{equation}
where the supremum is taken over all partitions such that $0=t_0<t_1<\cdots<t_n=t$.  

Solving SDE \eqref{SDEnu}  is known as the Skorokhod problem \cite{skorokhod1961stochastic,skorokhod1962stochastic}, which was introduced in 1961 to study an SDE with a reflecting diffusion process on half line. The multi-dimensional version of Skorokhod's problem was solved in \cite{tanaka1979stochastic}, where the domain $D$ was assumed to be convex. In \cite{lions1984stochastic} the authors relaxed the convexity assumption to domains satisfying certain admissibility conditions. 
Such admissibility assumption was subsequently removed in \cite{saisho1987stochastic} by applying the techniques used in \cite{tanaka1979stochastic}.

The self-consistent stochastic process associated with \eqref{PDE} has the form
\begin{equation}\label{SDE}
Y_t=Y_0+\int_{0}^tP_{Y_s}\left(-\nabla K\ast \mu_s(Y_s)-\nabla V(Y_s)\right)\ds, \quad t>0,\\
\end{equation}
where $\mu_t=Law(Y_t)$ is the probability distribution of $Y_t$ and the initial data $Y_0$ with distribution $\mu_0$ is chosen to be the same as in \eqref{SDEnu}.  Since  the randomness  of \eqref{SDE} only comes  from the initial data $Y_0$,  we can treat  it as a standard ODE for a random $\omega\in \Omega$ fixed. Well-posedness of \eqref{SDE} can be shown as in \cite[Lemma 2.4]{carrillo2014nonlocal} by using the theory of  differential inclusions \cite{filippov2013differential,edmond2006bv}. 

As a first step toward proving the zero-diffusion limit, we establish the well-posedness of the SDEs \eqref{SDEnu} and~\eqref{SDE} (Theorems \ref{wellSDEnu} and \ref{wellSDE}) and prove that the probability distribution $\mu_t^\nu$ of $Y_t^\nu$ is the weak solution to the aggregation-diffusion equation \eqref{PDEnu} (Theorem \ref{wellPDEnu}). 
These well-posedness results are proved under Assumption \ref{asum1} in Section~\ref{sect:zero-diff} for the interaction and external potentials $K$ and $V$. We note that well-posedness of~\eqref{SDEnu} has been shown in the literature for various types of potentials: in \cite{sznitman1984nonlinear} $\nabla K$ is assumed to be bounded and Lipschitz. In \cite{choi2016propagation} a specific type of $K$ with $\nabla K$ being bounded discontinuous was considered. In \cite{fetecau2018propagation} well-posedness of~\eqref{SDEnu} was shown when $K$ is the Newtonian potential. The particular type of potentials in our paper is not covered in the previous cases. Nonetheless, our method is similar to the ones used in~\cite{choi2016propagation, fetecau2018propagation} in the sense that we regularize the potentials first and prove the convergence of the regularized solutions. 

Next we prove our first main result concerning the zero-diffusion limit of solutions to \eqref{PDEnu} (Theorem \ref{mainthm1}). This is achieved by using a coupling method to show the convergence of $\mu_t^\nu$ to $\mu_t$ via the comparison of $Y_t^\nu$ with $Y_t$. The precise estimate is 
\begin{equation}
\mathcal{W}_2^2(\mu_t^\nu,\mu_t)
\leq  
\EE\left[|Y_t^\nu-Y_t|^2\right]
\leq O (\nu).
\end{equation}
Compared with the zero-diffusion limit established in~\cite{zhang2017continuity},
our result improves in two respects: i) a sharper (and believed to be optimal) convergence rate and ii) relaxed regularity assumptions on the interaction and external potentials. Specifically, the convergence rate established in \cite{zhang2017continuity} is $O(\nu^{\frac{1}{d+2}})$. This is obtained by bounding $\rho^\nu_t$, the density function of measure $\mu_t^\nu$, in its $L^\infty$-norm. 
In our case, we only need to work with $\mu^\nu_t$ as a probability measure so that $\mu^\nu_t(D) = 1$. This improves the convergence rate to $O(\nu)$; we also note here that the numerical results presented in our paper suggest that this is the optimal rate. The second improvement is to reduce the requirement on the interaction and external potentials: we assume $K$ and $V$ to be $C^1$ and $\lambda$-convex,
while in~\cite{zhang2017continuity} they are assumed to be  $C^2$.

In the second part of this article (Sections \ref{sect:particle} and \ref{sect:numerics}), we show the zero-diffusion limit on the discrete level and prove the mean-field limit of the particle system with the Brownian motion. This system approximates \eqref{PDEnu} and has the form
\begin{align}\label{Rparticlenu}
\begin{cases}
{X}_t^{\nu,i}={X}^i_0- \int_{0}^t\Bigl(\frac{1}{N-1}\sum\limits_{j\neq i}^N\nabla K ({X}_s^{\nu,i} \!-\! {X}_s^{\nu,j}) +\!\nabla V({X}_s^{\nu,i})\Bigr) \!\ds+\sqrt{2\nu}{B}_t^{\nu,i} \!-\!R_t^{\nu,i},\quad 1 \leq i \leq N, \,\, t>0, \\[3pt]
R_t^{\nu,i}=\int_0^tn(X_t^{\nu,i}) \, {\rm d} |R^{\nu,i}|_s,\quad |R^{\nu,i}|_t=\int_0^t\textbf{1}_{\partial D}(X_s^{\nu,i}) \, {\rm d}|R^{\nu,i}|_s \,,
\end{cases}
\end{align}
where the initial data $\{{X}^i_0\}_{i=1}^N$ are i.i.d. random variables with a common probability distribution $\mu_0(x)$. 
We note that the existence of solutions to system \eqref{Rparticlenu} has been shown in \cite[Theorem 1.1]{choi2016propagation}. 
The corresponding deterministic particle system approximating \eqref{PDE} introduced in \cite[Theorem 2.6]{carrillo2014nonlocal} has the form 
\begin{align}\label{Rparticle}
{X}_t^{i}={X}^i_0+\frac{1}{N-1}\sum\limits_{j\neq i}^N \int_{0}^tP_{X_s^i}\left(-\nabla K({X}_s^{i}-{X}_s^{j}) -\nabla V({X}_s^{i})\right)\ds,\quad i=1,\cdots, N, \quad t>0\,,
\end{align}
where the initial data $\{{X}^i_0\}_{i=1}^N$ are i.i.d. random variables with the same probability distribution $\mu_0(x)$. The well-posedness of~\eqref{Rparticle} and its mean-field limit are both established in~\cite{carrillo2014nonlocal}. 

Let $\munuX$ and  $\mu_t^{X}$ denote the empirical measures associated to \eqref{Rparticlenu} and \eqref{Rparticle}, respectively: 
\begin{align*}
   \munuX=\frac{1}{N}\sum\limits_{i=1}^N\delta_{X_t^{\nu,i}} \,,
\qquad
   \mu_t^{X} =\frac{1}{N}\sum\limits_{i=1}^N\delta_{X_t^{i}} \,.
\end{align*}
The main result (Theorem \ref{thmparticle}) in this part of the paper is to show the convergence of $\munuX$ to $\mu_t^{X}$ with the quantitative estimate:
\begin{equation}
	\EE\left[\mathcal{W}_\infty^2(\munuX,\mu_t^X)\right]\leq O(\nu).
\end{equation}
In addition, we show that the empirical measure $\munuX$ associated to the particle system \eqref{Rparticlenu} approximates the solution $\mu_t^\nu$ to equation \eqref{PDEnu} as the number $N$ of particles goes to infinity (see Theorem \ref{thmmean}). Such mean-field limit has been proved for various types of potentials.  
For example,  in the case of $\R^d$, the mean-field limits are justified for $K$ being the Newtonian potential \cite{HH1,HH2,garcia2017,fournier2015stochastic,huilearning}, the  Biot-Savart potential \cite{fournier2014propagation}, and a delta distribution \cite{chen2018modeling}. In the case of domains with boundaries, 
the mean-field limit is derived in \cite{sznitman1984nonlinear} for particle systems with reflecting boundary conditions and $\nabla K$ being bounded Lipschitz. In \cite{choi2016propagation} a particular type of bounded discontinuous interacting force was considered. More recently, in \cite{fetecau2018propagation} the authors justified the mean-field limit for the particle system interacting through the Newtonian potential. As for \textit{deterministic} particle methods (without the diffusion term), we refer the reader to \cite{craig2014blob,carrillo2017blob} (see also  \cite{chertock2017practical,jabin2017mean} and references therein for a comprehensive review).


Finally, within the discrete framework we perform careful numerical tests to verify numerically the convergence rate and to gain further insight on how small diffusion regularizes model \eqref{PDE}.
 
The paper is organized as follows. In Section \ref{sect:prelims}, we give a brief introduction to the Wasserstein metric and provide definitions of weak solutions to equations \eqref{PDEnu} and  \ref{PDE}. In Section \ref{sect:zero-diff}, we prove the main result on the zero-diffusion limit after showing the well-posedness of SDEs \eqref{SDEnu} and \eqref{SDE}. Section \ref{sect:particle} presents particle methods for equations  \eqref{PDEnu}  and \eqref{PDE}, where we further establish the zero-diffusion limit result at particle level. Section \ref{sect:numerics} is devoted to the numerical justification of the convergence rate.

\section{Preliminaries}
\label{sect:prelims}
\subsection{The $p$-Wasserstein space}
In order to show the convergence of weak solutions, we give a brief introduction on the topology of the $p$-Wasserstein space.
Consider the following space
\begin{eqnarray}\label{p1d}
\mathcal{P}_p(D)=\left\{f\in\mathcal{P}(D):~\int_{D}|x|^p {\rm d}f(x)<+\infty\right\}\,,
\end{eqnarray}
where $\mathcal{P}(D)$ denotes all the probability measures on $D$.
We denote the $p$-Wasserstein  distance in $\mathcal{P}_p(D)$ as follows
\begin{equation}\label{wass}
\mathcal {W}_p(f,g)=\left(\inf_{\pi\in\Lambda(f,~g)}\Big\{\int_{D\times D}|x-y|^p\rm d\pi(x,y)\Big\}\right)^{\frac{1}{p}}=\left(\inf_{X\sim f,Y\sim g}\Big\{\EE[|X-Y|^p]\Big\}\right)^{\frac{1}{p}},
\end{equation}
where $\Lambda(f,~g)$ is the set of joint probability measures on ${D}\times{D}$ with marginals $f$ and $g$ respectively and $(X,Y)$ are all possible couples of random variables with $f$ and $g$ as respective distributions.  We will also use the infinite Wasserstein distance $\mathcal {W}_\infty$ defined as
\begin{equation}\label{Winf}
\mathcal {W}_\infty(f,g):=\inf_{X\sim f,Y\sim g}\PP\mbox{-ess sup }|X-Y|=\inf_{\pi\in \Lambda(f,~g)}\pi\mathop{\mbox{-ess sup }}\limits_{(x,y)\in D\times D}|x-y|\,,
\end{equation}
where
\begin{align}\label{ppess}
\PP\mbox{-ess sup }|X-Y|:=\inf\big \{\lambda\geq 0: \PP(|X-Y|>\lambda)=0\big\}\,,
\end{align}
and
\begin{align}\label{piess}
\pi \mathop{\mbox{-ess sup }}\limits_{(x,y)\in D\times D} |x - y|:=\inf\big \{\lambda\geq 0: \pi(\{(x,y)\in D\times D: |x-y|> \lambda\})=0\big\}\,.
\end{align}
Moreover, it holds that for all $ p\leq q\leq \infty$,
\begin{equation}\label{wrela}
\mathcal {W}_p(f,g)\leq \mathcal {W}_q(f,g)\leq (\mbox{diam }D)^{\frac{q-1}{q}}\mathcal {W}_1(f,g)^{\frac{1}{q}},
\end{equation}
and $\mathcal {W}_q(f,g)\rightarrow \mathcal {W}_\infty(f,g)$ as $q\rightarrow \infty$.
We refer readers to books \cite{ambrosio2008gradient,villani2008optimal} for further background.

Moreover, in \cite[Theorem 6.18]{villani2008optimal}, it has been shown that for any $p \geq 1$, the space $\mathcal{P}_p(D)$ endowed with the metric $\mathcal{W}_p$ is a complete metric space. For bounded domains, the convergence in $(\mathcal{P}_p(D),\mathcal{W}_p)$ is equivalent to the weak-$\ast$ convergence in the sense that
\begin{align}\label{starconv}
\mathcal{W}_p(f_n,f)\rightarrow 0\Longleftrightarrow f_n\rightarrow f \mbox{ weak-}\ast \mbox{ in } \mathcal{P}(D) \,.
\end{align}

\subsection{Definitions of weak solutions}
Next we define the notion of weak measure-valued solutions for equations \eqref{PDEnu} and \eqref{PDE}. Recall that  a  curve $\mu(t)\in AC^2_{loc}((a,b);\mathcal{P}_2(D))$ is locally 2-absolutely continuous if there exists $m\in L^2_{loc}((a,b))$ such that
\begin{equation}
\mathcal{W}_2(\mu(t),\mu(s))\leq \int_{s}^{t} m(r){\rm d}r\,,\quad\mbox{for all }t,s\in(a,b)\mbox{ with }a<s\leq t<b\,.
\end{equation}

\begin{definition}\label{weaksolutionpde}
	For $\mu_0\in \mathcal{P}_2(D)$, a curve $\mu_t^\nu \in AC^2_{loc}([0,T];\mathcal{P}_2(D))$ is a weak solution to  \eqref{PDEnu} if
	it holds that
	\begin{align}\label{Sindensityfun}\nonumber
	\int_0^T \int_D \left[\partial_t \phi (t,x)-\la \nabla \phi (t,x), \nabla K\ast\mu^{\nu}_t(x) +\nabla V(x)\ra +\nu \Delta \phi (t,x)\ra
	\right]\rd \mu_t^{\nu} (x) \dt=0\,,\quad \mu_t^\nu|_{t=0} =\mu_0,
	\end{align}
	for any test function $\phi\in {C}_c^\infty(D\times (0,T))$ satisfying $\langle \nabla\phi,n \rangle=0$ on $\partial D$.
\end{definition}
The well-posedness of measure-valued solutions to equation \eqref{PDEnu} with Assumption \ref{asum1} specified in Section~\ref{sect:zero-diff} follows from the theory of gradient flows \cite[Theorem 11.2.8]{ambrosio2008gradient}. Note that by the regularizing effect of the Wassertein semigroup, the solution $\mu_t^\nu$ is actually absolutely continuous with respect to the Legesgue measure $\mathcal{L}^d$ for all $t>0$.

\begin{definition}\label{weaksolutionpdenu}
	For $\mu_0\in\mathcal{P}_2( D)$, a curve $\mu_t \in AC^2_{loc}([0,T];\mathcal{P}_2( D))$ is a weak solution to  \eqref{PDE} if
	\begin{equation*}
	P_x\left( -\nabla K\ast \mu_t-\nabla V\right)\in L_{loc}^1([0,T];L^2(\mu_t))\,,
	\end{equation*}
	and if holds that 
	\begin{align}\label{Sindensityfun'}\nonumber
	\int_0^T \int_D \left[\partial_t \phi (t,x)+\la \nabla \phi (t,x), P_x\left( -\nabla K\ast \mu_t(x)-\nabla V(x)\right)\ra
	\right]\rd \mu_t(x) \dt=0\,,\quad \mu_t|_{t=0} =\mu_0,
	\end{align}
	for any test function $\phi\in {C}_c^\infty( D\times (0,T))$, where supp $\mu_t(x)\subset D$ for all $t\in[0,T]$.
\end{definition}
The well-posedness of weak solutions to the plain aggregation model \eqref{PDE} was proved in \cite[Theorem 1.5 and Theorem 1.6]{carrillo2014nonlocal} with Assumption \ref{asum1} for $K, V$ given in Section~\ref{sect:zero-diff} .

\section{Zero-diffusion limit}
\label{sect:zero-diff}
In this section we prove the convergence of weak solutions through the coupling method. The main assumptions on the interaction potential $K$, external potential $V$  and domain $D$ are
\begin{assum}\label{asum1}
	\quad
	\begin{enumerate}
		\item $D\subset \RR^d$ is bounded and convex  and $\partial D\in C^1$. We denote by $D-D:=\{x-y|~x,y\in D\}$.
		\item $K(x)=K(-x)$ for  all $x\in\RR^d$.
		\item $K\in C^1(D-D)$ and it  is $\lambda_K$-convex on $ D- D$ for some $\lambda_K\in\RR$, which implies that 
		\begin{equation}\label{lamcon}
		\la \nabla K(x)-\nabla K(y),x-y\ra\geq \lambda_K |x-y|^2\,.
		\end{equation}
		\item $V\in C^1(D)$ and it  is $\lambda_V$-convex on $D$ for some $\lambda_V\in\RR$, which implies that 
		\begin{equation}\label{lamcon1}
		\la  \nabla V(x)- \nabla V(y),x-y\ra\geq \lambda_V |x-y|^2\,.
		\end{equation}
		Furthermore, we denote
		\begin{equation}\label{lambda}
		\lambda_K^-:=\min\{0,\lambda_K\}\leq 0,\quad \lambda_V^-:=\min\{0,\lambda_V\}\leq 0.
		\end{equation}
	\end{enumerate}
\end{assum}
\begin{rmk}
	Recall that a funtion $f\in C^1(D)$ is $\lambda$- convex on a convex set $D$ if 
	\begin{equation}
	f(y)\geq f(x)+\la\nabla f(x),y-x\ra+\frac{\lambda}{2}|y-x|^2
	\qquad \text{for any $x,y\in D$.}
	\end{equation}
	Note that $\lambda$- convexity is a weaker regularity than $C^2$. For example, the function $f(x) = |x|^{\frac{3}{2}}$ is $\lambda$- convex but not $C^2$ on $[0, 1]$.
\end{rmk}

\subsection{ Well-posedness of the self-consistent stochastic processes}

Next we prove the well-posedness of SDE \eqref{SDEnu}. To this end, let $J(x)$ be a blob function such that
\begin{align*}
J \in C^\infty(\mathbb{R}^d) \,,
\quad
J \geq 0 \,,
\quad
\mbox{supp}\,J(x)\subset {B}(0, 1)\,,
\quad
\int_{{B}(0, 1)}J(x) \dx=1\,.
\end{align*}
For any $\Eps > 0$, define $J_\varepsilon(x)=\frac{1}{\varepsilon^d}J(\frac{x}{\varepsilon})$ and
\begin{align} \label{def:F-Eps}
K_\varepsilon(x)= J_\varepsilon\ast K(x),\quad V_\varepsilon(x)= J_\varepsilon\ast V(x) \,.
\end{align}
Replacing $K$ by $K_\Eps$ in~\eqref{PDEnu}, one has the regularized aggregation equation
\begin{align}\label{rePDEnu}
\begin{cases}
\partial_t\mu^{\nu,\varepsilon}_t =\nu\triangle \mu^{\nu,\varepsilon}_t+\nabla \cdot[\mu^{\nu,\varepsilon}_t (\nabla K_\varepsilon\ast \mu^{\nu,\varepsilon}_t+\nabla V_\varepsilon)] \,, 
\quad &t>0\,,\\[2pt]
\mu^{\nu,\varepsilon}_t|_{t=0}=\mu_0\,,
 \\[2pt]
\langle \nu\nabla\mu^{\nu,\varepsilon}_t+\mu^{\nu,\varepsilon}_t (\nabla K_\varepsilon \ast \mu^{\nu,\varepsilon}_t+\nabla V_\varepsilon),n\rangle=0 \,,
\quad &\mbox{on }\partial D \,.
\end{cases}
\end{align}
Since $K_\varepsilon,V_\varepsilon\in C^2(D)$ for any fixed $\varepsilon$, the regularized PDE \eqref{rePDEnu} has a unique weak solution $\mu^{\nu,\varepsilon}_t\in AC^2_{loc}([0,T];\mathcal{P}_2(D))$ (\cite[Theorem 4.1, Theorem 4.2]{zhang2017continuity}).

The self-consistent process underlying equation \eqref{rePDEnu} is
\begin{equation}\label{RSDEnu}
\left\{\begin{array}{l}
Y^{\nu,\varepsilon}_t=Y_0-\int_{0}^t\left(\nabla K_\varepsilon\ast\mu_s^{\nu,\varepsilon}(Y^{\nu,\varepsilon}_s) +\nabla V_\varepsilon(Y^{\nu,\varepsilon}_s)\right) \ds+\sqrt{2\nu}{B}_t-\tilde R_t^{\nu,\varepsilon}, \quad t>0,\\
\tilde R_t^{\nu,\varepsilon}=\int_0^tn(Y^{\nu,\varepsilon}_s) \, {\rm d}|\tilde R^{\nu,\varepsilon}|_s,\quad |\tilde R^{\nu,\varepsilon}|_t=\int_0^t\textbf{1}_{\partial D}(Y^{\nu,\varepsilon}_s) \, {\rm d}|\tilde R^{\nu,\varepsilon}|_s,\\
\end{array}\right.
\end{equation}
where $\mu_t^{\nu,\varepsilon}=Law(Y_t^{\nu,\varepsilon})$ is the law of $Y_t^{\nu,\varepsilon}$. With the help of the well-posedness of the equation \eqref{rePDEnu}, we also have the existence and uniqueness of the SDE \eqref{RSDEnu}. Moreover, we prove that the law of $Y_t^{\nu,\varepsilon}$ is exactly the weak solution to the PDE \eqref{rePDEnu}.
\begin{proposition}\label{pdetosde}
	Let $\mu^{\nu,\varepsilon}_t\in AC^2_{loc}([0,T];\mathcal{P}_2(D))$ be the unique weak solution to the PDE \eqref{rePDEnu} with the initial data $\mu_0$. Assume that $Y_0$ has the distribution $\mu_0$ and $(B_t)_{t\geq 0}$ is a Brownian motion independent of $Y_0$. Then the SDE \eqref{SDEnu} has a unique solution $(Y_t^{\nu,\varepsilon},\tilde R_t^{\nu,\varepsilon})$ with $Law(Y_t^{\nu,\varepsilon})=\mu_t^{\nu,\varepsilon}$.
\end{proposition}
\begin{proof}
	Let $\mu^{\nu,\varepsilon}_t$ be the unique solution of \eqref{rePDEnu} with the initial data $\mu^{\nu,\varepsilon}_0=\mu_0$. For any fixed $\varepsilon$, $\nabla K_\varepsilon\ast \mu^{\nu,\varepsilon}_t$ and $\nabla V_\varepsilon$ are bounded and Lipschitz. Therefore, according to  \cite[Theorem 1.1]{sznitman1984nonlinear} or \cite[Lemma 3.1]{choi2016propagation}, for any $T > 0$,  the SDE
	\begin{equation}
	\left\{\begin{array}{l}
	\overline Y^{\nu,\varepsilon}_t=Y_0-\int_{0}^t\left(\nabla K_\varepsilon\ast\mu^{\nu,\varepsilon}_s(\overline Y^{\nu,\varepsilon}_s) +\nabla V_\varepsilon(\overline Y^{\nu,\varepsilon}_s)\right) \ds+\sqrt{2\nu}{B}_t-\tilde R_t^{\nu,\varepsilon}, \quad t>0,\\
	\tilde R_t^{\nu,\varepsilon}=\int_0^tn(\overline Y^{\nu,\varepsilon}_s) \, {\rm d}|\tilde R^{\nu,\varepsilon}|_s,\quad |\tilde R^{\nu,\varepsilon}|_t=\int_0^t\textbf{1}_{\partial D}(\overline Y^{\nu,\varepsilon}_s) \, {\rm d}|\tilde R^{\nu,\varepsilon}|_s,\\
	\end{array}\right.
	\end{equation}
	where $Law(Y_0)=\mu_0$, has a unique solution $(\overline Y_t^{\nu,\varepsilon},\tilde R_t^{\nu,\varepsilon})$  up to time $T > 0$.  Here $\mu^{\nu,\varepsilon}_t$ is prescribed as the weak solution to the PDE and we do not know whether or not it is the distribution of $\overline Y_t^{\nu,\varepsilon}$ yet.
	
	Let $Law(\overline Y_t^{\nu,\varepsilon})=\overline \mu_t^{\nu,\varepsilon}$, then we have from  It\^{o}'s formula \cite{brent} for any test function $\phi\in {C}_c^\infty(D\times (0,T))$ satisfying $\la \nabla\phi(t,x),n(x)\ra =0$ on $x\in \partial D$ that
	\begin{align}\label{ito1}
	\phi (t,\overline Y_t^{\nu,\varepsilon})-\phi (0,Y_0)&=\int_0^t \left[\partial_s \phi (s,\overline Y_s^{\nu,\varepsilon})-\la \nabla \phi (s, \overline Y_s^{\nu,\varepsilon}), \nabla K_\varepsilon\ast\mu^{\nu,\varepsilon}_s(\overline Y^{\nu,\varepsilon}_s) +\nabla V_\varepsilon(\overline Y^{\nu,\varepsilon}_s)\ra +\nu \Delta \phi (s,\overline Y_s^{\nu,\varepsilon})
	\right]\ds\notag \\
	&\quad -\int_0^t \la \nabla \phi (s,\overline Y_s^{\nu,\varepsilon}),  \rd \tilde R_s^{\nu,\varepsilon}\ra +\sqrt{2\nu}\int_0^t \la \nabla \phi (s,\overline Y_s^{\nu,\varepsilon}), \rd{B}_s\ra,
	\end{align}
	Using the boundary condition one has
	\begin{equation*}
	\int_0^t \la \nabla \phi (s,\overline Y_s^{\nu,\varepsilon}),  \rd \tilde R_s^{\nu,\varepsilon}\ra =\int_0^t \la \nabla \phi (s,\overline Y_s^{\nu,\varepsilon}), n(\overline Y_s^{\nu,\varepsilon}) \ra  \rd |\tilde R^{\nu,\varepsilon}|_s=0.
	\end{equation*}
	By taking expectation on \eqref{ito1}, we get
	\begin{align}\label{weak}
	\int_0^T \int_D \left[\partial_s \phi (s,x)-\la \nabla \phi (s,x), \nabla K_\varepsilon\ast\mu^{\nu,\varepsilon}_s(x) +\nabla V_\varepsilon(x)\ra +\nu \Delta \phi (s,x)\ra
	\right]\rd \overline \mu_t^{\nu,\varepsilon} (x) \ds=0\,,
	\end{align}
	where we used the fact that
	\begin{equation*}
	\phi (T,\overline Y_t^{\nu,\varepsilon})=\phi (0,Y_0)=0.
	\end{equation*}
	This  means that $\overline \mu_t^{\nu,\varepsilon}$ is a weak solution to the following linear PDE
	\begin{align}\label{linearrePDEnu}
	\begin{cases}
	\partial_t\overline\mu^{\nu,\varepsilon}_t =\nu\triangle \overline \mu^{\nu,\varepsilon}_t+\nabla \cdot[\overline\mu^{\nu,\varepsilon}_t (\nabla K_\varepsilon\ast \mu^{\nu,\varepsilon}_t+\nabla V_\varepsilon)] \,, 
	\quad & t>0\,,\\[2pt]
	\overline\mu^{\nu,\varepsilon}_t|_{t=0}=\mu_0\,,
	\\[2pt]
	\langle \nu\nabla\overline\mu^{\nu,\varepsilon}_t+\overline\mu^{\nu,\varepsilon}_t (\nabla K_\varepsilon \ast \mu^{\nu,\varepsilon}_t+\nabla V_\varepsilon),n\rangle=0 \,,
	\quad &\mbox{on }\partial D \,.
	\end{cases}
	\end{align}
By the uniqueness of the linear PDE~\eqref{linearrePDEnu} and the assumption that  $\mu_t^{\nu,\varepsilon}$ is a unique solution to PDE \eqref{rePDEnu}, we obtain that $\overline \mu_t^{\nu,\varepsilon}=\mu_t^{\nu,\varepsilon}$. Hence the unique solution to SDE \eqref{RSDEnu} is given by
	$(\overline Y_t^{\nu,\varepsilon},\tilde R_t^{\nu,\varepsilon})$ with $Law(\overline Y_t^{\nu,\varepsilon})=\overline \mu_t^{\nu,\varepsilon}=\mu_t^{\nu,\varepsilon}$.
\end{proof}

The following lemma plays a crucial role in the proof of the existence for the nonlinear SDE \eqref{SDEnu}, 
\begin{lem}\label{lmconvex}
	Let $X_1, X_2$ be two random variables with distributions $\mu_1, \mu_2$ respectively (${X}_1$  and ${X}_2$ may not be independent). Suppose that $K$ satisfies  Assumption \ref{asum1}. Then it holds that
	\begin{equation}
	\EE\left[-\viint{X_1-X_2}{\nabla K\ast\mu_1(X_1)-\nabla K\ast\mu_2(X_2)}\right]\leq -\lambda_K^-\EE[|X_1-X_2|^2],
	\end{equation}
	where $\lambda_K^-$ is defined in \eqref{lambda}.
\end{lem}
\begin{proof}
	We introduce the joint distribution of $X_1$ and $X_2$ as $\pi:=Law(X_1,X_2)$. Then 
	\begin{align}
	&\EE\left[-\viint{X_1-X_2}{\nabla K\ast\mu_1(X_1)-\nabla K\ast\mu_2(X_2)}\right]\notag\\
	=&- \iint_{D\times D}\viint{x_1-x_2}{\nabla K\ast\mu_1(x_1)-\nabla K\ast\mu_2(x_2)}\rd\pi(x_1,x_2) \notag\\
	=&-\frac{1}{2}\iint_{D\times D}\iint_{D\times D}\viint{x_1-x_2+y_2-y_1}{\nabla K(x_1-y_1)-\nabla K(x_2-y_2)}\rd\pi(y_1,y_2) \rd\pi(x_1,x_2) \notag\\
	\leq&-\frac{\lambda_K}{2}\iint_{D\times D}\iint_{D\times D}|x_1-y_1-x_2+y_2|^2\rd\pi(y_1,y_2) \rd\pi(x_1,x_2) \label{uselamcon}\\
	\leq &-\lambda_K^-\iint_{D\times D}|x_1-x_2|^2\rd\pi(x_1,x_2)\notag\\
	=&-\lambda_K^-\EE[|X_1-X_2|^2] ,\notag
	\end{align}
	where we have used the symmetry and the $\lambda_K$-convexity of $K$ given in Assumption \ref{asum1}. 
\end{proof}

Next we show the well-posedness of SDE \eqref{SDEnu}.
\begin{thm}\label{wellSDEnu} Let $K$, $V$ and $D$ satisfy Assumption \ref{asum1}. Assume that $Y_0$ has the distribution $\mu_0$ and $(B_t)_{t\geq 0}$ is a Brownian motion independent of $Y_0$. For any $T>0$, there exists only one strong solution $(Y_t^\nu,\tilde R_t^\nu)$ to the nonlinear SDE \eqref{SDEnu} on $[0,T]$ with $Law(Y_t^\nu)=\mu_t^\nu\in AC^2_{loc}([0,T];\mathcal{P}_2(D))$. Moreover, if $Y_t^{\nu,i}$, $i=1,2$ are two solutions to \eqref{SDEnu} with the initial data $Y_0^{i}$ respectively, then we have the following stability estimate
	\begin{equation}\label{SDEstab}
	\EE[|Y_t^{\nu,1}-Y_t^{\nu,2}|^2]\leq \left(1-2t(\lambda_K^-+\lambda_V^-)e^{-2t(\lambda_K^-+\lambda_V^-)} \right)\EE[|Y_0^{1}-Y_0^{2}|^2]\quad \mbox{for all }t\in[0,T],
	\end{equation}
	where $\lambda_K^-$  and $\lambda_V^-$ are  defined in \eqref{lambda}.
\end{thm}
\begin{proof} 
	$\bullet$\textit{Step 1 (Cauchy estimate):} 	Let $Y^{\nu,\Eps}_t$ and $Y^{\nu,\Eps'}_t$ be the solutions to the regularized nonlinear SDE \eqref{RSDEnu}.
	Applying It\^{o}'s formula to $Y^\Eps_t - Y^{\Eps'}_t$, we have
	\begin{align*} 
	|Y_t^{\nu,\varepsilon}-Y_t^{\nu,\varepsilon'}|^2
	&=-2\int_0^t\la Y_s^{\nu,\varepsilon}-Y_s^{\nu,\varepsilon'},\nabla K_\varepsilon\ast\mu_s^{\nu,\varepsilon}(Y_s^{\nu,\varepsilon})-\nabla K_{\varepsilon'}\ast\mu_s^{\nu,\varepsilon'}(Y_s^{\nu,\varepsilon'})\ra \ds  \nn
	\\ 
	&\quad -2\int_0^t\la Y_s^{\nu,\varepsilon}-Y_s^{\nu,\varepsilon'},\nabla V_\varepsilon(Y_s^{\nu,\varepsilon})-\nabla V_{\varepsilon'}(Y_s^{\nu,\varepsilon'})\ra \ds  \nn\\
	&\quad 
	-2\int_0^t\la Y_s^{\nu,\varepsilon}-Y_s^{\nu,\varepsilon'},n(Y_s^{\nu,\varepsilon})\ra {\rm d}|\tilde R^{\nu,\varepsilon}|_s-2\int_0^t
	\la Y_s^{\nu,\varepsilon'}-Y_s^{\nu,\varepsilon},n(Y_s^{\nu,\varepsilon'})\ra {\rm d} |\tilde R^{\nu,\varepsilon'}|_s \,.
	\end{align*}
	By the convexity of the domain $D$, one has
	\begin{align}\label{conxdomain}
	(x-y)\cdot n(x)\geq 0\quad \mbox{for any } x\in\partial D\mbox{ and } y\in D.
	\end{align}
	Consequently, 
	\begin{align*}
	\la Y_s^{\nu,\varepsilon}-Y_s^{\nu,\varepsilon'},n(Y_s^{\nu,\varepsilon})\ra 
	\geq 0,
	\quad 
	{\rm d}|\tilde R^{\nu,\varepsilon}|_s \mbox{ almost surely}\,
	\end{align*}
	and
	\begin{align*}
	\la Y_s^{\nu,\varepsilon'}-Y_s^{\nu,\varepsilon},n(Y_s^{\nu,\varepsilon'})\ra  
	\geq 0,
	\quad 
	{\rm d}|\tilde R^{\nu,\varepsilon'}|_s \mbox{ almost surely} \,.
	\end{align*}
	Therefore, one has
	\begin{align*} 
	|Y_t^{\nu,\varepsilon}-Y_t^{\nu,\varepsilon'}|^2
	&\leq-2\int_0^t\la Y_s^{\nu,\varepsilon}-Y_s^{\nu,\varepsilon'},\nabla K_\varepsilon\ast\mu_s^{\nu,\varepsilon}(Y_s^{\nu,\varepsilon})-\nabla K_{\varepsilon'}\ast\mu_s^{\nu,\varepsilon'}(Y_s^{\nu,\varepsilon'})\ra \ds  \nn
	\\ 
	&\quad -2\int_0^t\la Y_s^{\nu,\varepsilon}-Y_s^{\nu,\varepsilon'},\nabla V_\varepsilon(Y_s^{\nu,\varepsilon})-\nabla V_{\varepsilon'}(Y_s^{\nu,\varepsilon'})\ra \ds 	\\
	&=: I_1+I_2.
	\end{align*}
The first term $I_1$ satisfies
	\begin{align*} 
	I_1&=-2\int_0^t\langle Y_s^{\nu,\varepsilon}-Y_s^{\nu,\varepsilon'},\nabla K\ast\mu_s^{\nu,\varepsilon}(Y_s^{\nu,\varepsilon})-\nabla K\ast\mu_s^{\nu,\varepsilon'}(Y_s^{\nu,\varepsilon'})\rangle \ds \\
	&\quad -	2\int_0^t\langle Y_s^{\nu,\varepsilon}-Y_s^{\nu,\varepsilon'},\nabla K_\varepsilon\ast\mu_s^{\nu,\varepsilon}(Y_s^{\nu,\varepsilon})-\nabla K\ast\mu_s^{\nu,\varepsilon}(Y_s^{\nu,\varepsilon})\rangle \ds \\
	&\quad -	2\int_0^t\langle Y_s^{\nu,\varepsilon}-Y_s^{\nu,\varepsilon'},\nabla K\ast\mu_s^{\nu,\varepsilon'}(Y_s^{\nu,\varepsilon'})-\nabla K_{\varepsilon'}\ast\mu_s^{\nu,\varepsilon'}(Y_s^{\nu,\varepsilon'})\rangle \ds \\
	&\leq -2\int_0^t\langle Y_s^{\nu,\varepsilon}-Y_s^{\nu,\varepsilon'},\nabla K\ast\mu_s^{\nu,\varepsilon}(Y_s^{\nu,\varepsilon})-\nabla K\ast\mu_s^{\nu,\varepsilon'}(Y_s^{\nu,\varepsilon'})\rangle \ds \\
	&\quad+C t \sup\limits_{s\in[0,t]}\norm{\nabla K_\varepsilon\ast\mu_s^{\nu,\varepsilon}-\nabla K\ast\mu_s^{\nu,\varepsilon}}_\infty+C t\sup\limits_{s\in[0,t]}\|\nabla K_{\varepsilon'}\ast\mu_s^{\nu,\varepsilon'}-\nabla K\ast\mu_s^{\nu,\varepsilon'}\|_\infty,
	\end{align*}
	where $C$ only depends on $D$. 
	For $I_2$, we have
	\begin{align*}
	I_2&=-2\int_0^t\langle Y_s^{\nu,\varepsilon}-Y_s^{\nu,\varepsilon'},\nabla V(Y_s^{\nu,\varepsilon})-\nabla V(Y_s^{\nu,\varepsilon'})\rangle \ds -	2\int_0^t\langle Y_s^{\nu,\varepsilon}-Y_s^{\nu,\varepsilon'},\nabla V_\varepsilon(Y_s^{\nu,\varepsilon})-\nabla V(Y_s^{\nu,\varepsilon})\rangle \ds \\
	&\quad -	2\int_0^t\langle Y_s^{\nu,\varepsilon}-Y_s^{\nu,\varepsilon'},\nabla V(Y_s^{\nu,\varepsilon'})-\nabla V_{\varepsilon'}(Y_s^{\nu,\varepsilon'})\rangle \ds \\
	&\leq -2\lambda_V\int_0^t |Y_s^{\nu,\varepsilon}-Y_s^{\nu,\varepsilon'}|^2\ds +C t \norm{\nabla V_\varepsilon-\nabla V}_\infty+C t \norm{\nabla V_{\varepsilon'}-\nabla V}_\infty\,,
	\end{align*}
	where $C$ only depends on $D$. Here we also used the $\lambda_V$-convexity of $V$. By the continuity of $\nabla K$  and $\nabla V$, we have for all $s \in [0, t]$,
	\begin{equation*}
	\norm{\nabla K_\varepsilon\ast\mu_s^{\nu,\varepsilon}-\nabla K\ast\mu_s^{\nu,\varepsilon}}_\infty\leq \norm{\nabla K_\varepsilon-\nabla K}_\infty\norm{\mu_s^{\nu,\varepsilon}}_1\rightarrow 0,\quad \mbox{as }\varepsilon\rightarrow 0\,,
	\end{equation*}	
	and
	\begin{equation}\label{Vlimit}
	\norm{\nabla V_\varepsilon-\nabla V}_\infty\rightarrow 0,\quad \mbox{as }\varepsilon\rightarrow 0\,,
	\end{equation}
Hence,
	\begin{align}\label{cau}
	|Y_t^{\nu,\varepsilon}-Y_t^{\nu,\varepsilon'}|^2&\leq  -2\int_0^t\langle Y_s^{\nu,\varepsilon}-Y_s^{\nu,\varepsilon'},\nabla K\ast\mu_s^{\nu,\varepsilon}(Y_s^{\nu,\varepsilon})-\nabla K\ast\mu_s^{\nu,\varepsilon'}(Y_s^{\nu,\varepsilon'})\rangle \ds \notag\\
	&\quad -2\lambda_V \int_0^t|Y_s^{\nu,\varepsilon}-Y_s^{\nu,\varepsilon'}|^2\ds
	+C(\varepsilon,\varepsilon')t\,,
	\end{align}
	where $C(\varepsilon,\varepsilon')$ depends on $D$, $\varepsilon$ and $\varepsilon'$, and
	\begin{equation*}
	C(\varepsilon,\varepsilon')\rightarrow0\,,\quad \mbox{as }\varepsilon,\varepsilon'\rightarrow 0\,.
	\end{equation*}
	Taking expectation on \eqref{cau} and applying Lemma \ref{lmconvex}, we have
	\begin{equation}
	\EE\left[	|Y_t^{\nu,\varepsilon}-Y_t^{\nu,\varepsilon'}|^2\right]\leq (-2\lambda_K^--2\lambda_V^-)\int_0^t\EE\left[|Y_s^{\nu,\varepsilon}-Y_s^{\nu,\varepsilon'}|^2\right]ds+C(\varepsilon,\varepsilon')t\,.
	\end{equation}
	Applying Gronwall's inequality, one concludes that
	\begin{equation}
	\sup\limits_{t\in[0,T]}\EE\left[	|Y_t^{\nu,\varepsilon}-Y_t^{\nu,\varepsilon'}|^2\right] \leq C(\varepsilon,\varepsilon')T(1-2(\lambda_K^-+\lambda_V^-)T)e^{-2(\lambda_K^-+\lambda_V^-)T}\rightarrow0\,,
	\quad
	\mbox{ as }\varepsilon,\varepsilon'\rightarrow0 \,.
	\end{equation}
	
	$\bullet$\textit{Step 2 (Passing to the limit):} 	It follows from \textit{Step 1} that there exists a limiting stochastic process $Y^\nu_t \in C([0,T],D)$ such that \begin{align}\label{caucy1}
	\sup\limits_{t\in[0,T]}\EE\left[|Y_t^{\nu,\varepsilon}-Y_t^{\nu}|^2\right] \rightarrow0
	\quad
	\mbox{ as }\varepsilon\rightarrow0 \,.
	\end{align}
	Let $\mu_t^\nu$ be the distribution of $Y_t^\nu$ . Then
	\begin{align}\label{limitmu}
	\sup\limits_{t\in[0,T]} \mathcal{W}_2^2(\mu_t^{\nu,\varepsilon},\mu_t^\nu)\leq\sup\limits_{t\in[0,T]}\EE\left[|Y_t^{\nu,\varepsilon}-Y_t^{\nu}|^2\right] \rightarrow 0 \quad
	\mbox{ as }\varepsilon\rightarrow0\,.
	\end{align}
    This implies $\mu_t^\nu\in AC^2_{loc}([0,T];\mathcal{P}_2(D))$.
	
	Next we show that there exists a reflective process $\widetilde R_t^\nu$ such that $(Y_t^\nu, \widetilde R_t^\nu)$ is a strong solution to the SDE~\eqref{SDEnu}. 
	Define the process $\tilde R_t^\nu$ as
	\begin{equation}
	\tilde R_t^\nu
	:= - \vpran{Y_t^\nu-Y_0+\int_0^t \nabla K\ast\mu_s^\nu(Y_s^\nu)+\nabla V(Y_s^\nu)\ds-\sqrt{2\nu} B_t}\, ,
	\end{equation}
	and recall that 
	\begin{equation}
	\tilde R_t^{\nu,\varepsilon}
	= - \vpran{Y_t^{\nu,\varepsilon}-Y_0+\int_0^t \nabla K_\varepsilon\ast\mu_s^{\nu,\varepsilon}(Y_s^{\nu,\varepsilon})+\nabla V_\varepsilon(Y_s^{\nu,\varepsilon})\ds-\sqrt{2\nu} B_t}\, .
	\end{equation}
We want to show that
	\begin{equation}\label{estR}
	\tilde R_t^{\nu,\varepsilon}\stackrel{a.s.}{\longrightarrow} \tilde R_t^\nu\quad \mbox{as } \varepsilon\rightarrow 0.
	\end{equation}
	To this end, one first makes the splitting such that
	\begin{align*}
	|\nabla V_\varepsilon(Y_s^{\nu,\varepsilon})-\nabla V(Y_s^\nu)|\leq 	|\nabla V_\varepsilon(Y_s^{\nu,\varepsilon})-\nabla V(Y_s^{\nu,\varepsilon})|+|\nabla V(Y_s^{\nu,\varepsilon})-\nabla V(Y_s^\nu)|=:I_3+I_4.
	\end{align*}
	Then $I_3\rightarrow 0$ by \eqref{Vlimit}. It follows from \eqref{caucy1} that there exists a subsequence (independent of $s$) of $Y_s^{\nu,\varepsilon}$ (without relabeling $\varepsilon$) such that
	\begin{equation}
	Y_s^{\nu,\varepsilon}\stackrel{a.s.}{\longrightarrow} Y_s^\nu\quad \mbox{as } \varepsilon\rightarrow 0\,.
	\end{equation}
This leads to $I_4\stackrel{a.s.}{\longrightarrow}0$ since $\nabla V$ is continuous. Hence, we have
	\begin{equation}\label{estV}
	\nabla V_\varepsilon(Y_s^{\nu,\varepsilon})\stackrel{a.s.}{\longrightarrow}\nabla V(Y_s^\nu)\quad \mbox{as } \varepsilon\rightarrow 0.
	\end{equation}
	Similarly we have the bound
	\begin{align*}
	&|\nabla K_\varepsilon\ast\mu_s^{\nu,\varepsilon}(Y_s^{\nu,\varepsilon})-\nabla K\ast\mu_s^\nu(Y_s^\nu)|\\
	\leq &|\nabla K_\varepsilon\ast\mu_s^{\nu,\varepsilon}(Y_s^{\nu,\varepsilon})-\nabla K\ast\mu_s^{\nu,\varepsilon}(Y_s^{\nu,\varepsilon})|+ |\nabla K\ast\mu_s^{\nu,\varepsilon}(Y_s^{\nu,\varepsilon})-\nabla K\ast\mu_s^{\nu,\varepsilon}(Y_s^\nu)|\\
	&+|\nabla K\ast\mu_s^{\nu,\varepsilon}(Y_s^{\nu})-\nabla K\ast\mu_s^\nu(Y_s^{\nu})|=:J_1+J_2+J_3.
	\end{align*}
	Then $J_1\rightarrow 0$ uniformly and $J_2\stackrel{a.s.}{\longrightarrow}0$ by the similar argument for $I_3$ and $I_4$. Moreover $J_3\rightarrow 0$ by \eqref{starconv} and \eqref{limitmu}. Therefore, we have
	\begin{equation}\label{estK}
	\nabla K_\varepsilon\ast\mu_s^{\nu,\varepsilon}(Y_s^{\nu,\varepsilon})\stackrel{a.s.}{\longrightarrow}\nabla K\ast\mu_s^\nu(Y_s^\nu)\quad \mbox{as } \varepsilon\rightarrow 0.
	\end{equation}
Hence \eqref{estR} follows from \eqref{estK} and \eqref{estV}. We are left to check the properties of $\tilde R_t^\nu$:
	\begin{equation}
	\tilde R_t^\nu=\int_0^tn(Y^{\nu}_s) \, {\rm d}|\tilde R^\nu|_s\,,\quad |\tilde R^\nu|_t=\int_0^t\textbf{1}_{\partial D}(Y^{\nu}_s) \, {\rm d}|\tilde R^\nu|_s\,.
	\end{equation}
	These can be verified in the same way as in \cite[Step \textbf{B} on page 13]{choi2016propagation}. The details are omitted here. 
	
	$\bullet$\textit{Step 3 (Stability estimate):} Let $Y_t^{\nu,i}$, $i=1,2$ be two processes obtained as about with initial data $Y_0^{i}$ respectively and $Law(Y_t^{\nu,i})=\mu_t^{\nu,i}$. Similar argument as \textit{Step 1} yields that
	\begin{equation}
	\EE\left[	|Y_t^{\nu,1}-Y_t^{\nu,2}|^2\right]\leq  -2(\lambda_K^-+\lambda_V^-)\int_0^t\EE\left[|Y_s^{\nu,1}-Y_s^{\nu,1}|^2\right]\ds+\EE\left[	|Y_0^{1}-Y_0^{2}|^2\right],
	\end{equation}
	which implies the stability estimate \eqref{SDEstab} by Gronwall's inequality.
\end{proof}

As a direct result from Theorem \ref{wellSDEnu}, one has the well-posedness for the PDE \eqref{PDEnu}.
\begin{thm}\label{wellPDEnu}
	Let $K$, $V$ and $D$ satisfy Assumption \ref{asum1}. For any $T>0$, there exists an unique curve $\mu_t^\nu\in AC^2_{loc}([0,T];\mathcal{P}_2(D))$ that solves equation \eqref{PDEnu} weakly. Moreover, if $\mu_t^{\nu,i}$, $(i=1,2)$ are two solutions to \eqref{PDEnu}  with the initial data $\mu_0^{i}$ respectively, then the following Dobrushin's type  stability estimate holds
	\begin{equation}\label{PDEstab}
	\mathcal{W}_2^2(\mu_t^{\nu,1},\mu_t^{\nu,2})\leq \left(1-2t(\lambda_K^-+\lambda_V^-)e^{-2t(\lambda_K^-+\lambda_V^-)} \right)\mathcal{W}_2^2(\mu_0^{\nu,1},\mu_0^{\nu,2})\quad \quad \mbox{for all }t\in[0,T]\,,
	\end{equation}
	where $\lambda_K^-$  and $\lambda_V^-$ are  defined in \eqref{lambda}.
\end{thm}
\begin{proof}
	Assume that $(Y_t^\nu,\tilde R_t^\nu)$ is the unique solution to SDE \eqref{SDEnu} obtained in Theorem \ref{wellSDEnu}. Let $\mu_t^\nu\in AC^2_{loc}([0,T];\mathcal{P}_2(D))$ be the distribution of $Y_t^\nu$, then a direct application of  It\^{o}'s formula implies that $\mu_t^\nu$ is a weak solution to PDE \eqref{PDEnu}. The Dobrushin's type  stability \eqref{PDEstab} follows from the stability of $Y_t^\nu,$ in \eqref{SDEstab} and the fact that
	\begin{equation*}
	\mathcal{W}_2^2(\mu_t^{\nu,1},\mu_t^{\nu,2})\leq \EE\left[	|Y_t^{\nu,1}-Y_t^{\nu,2}|^2\right],
	\end{equation*}
	where $Y_t^{\nu,i}$, $i=1,2$ are two solutions to \eqref{SDEnu} with the initial data $Y_0^{i}$ respectively. This type of coupling argument has been used in \cite[Theorem 1.2]{choi2016propagation} and \cite[Theorem 2.6]{fetecau2018propagation}.
\end{proof}

Next theorem gives the well-posedness of SDE \eqref{SDE}.
\begin{thm}\label{wellSDE} Let $K$, $V$ and $D$ satisfy Assumption \ref{asum1}. Assume that $Y_0$ has the distribution $\mu_0$ and $(B_t)_{t\geq 0}$ is a Brownian motion independent of $Y_0$.
	For any $T>0$, there exists only one strong solution $Y_t$ to the nonlinear SDE \eqref{SDE} on $[0,T]$. Moreover, if $Y_t^{i}$, $(i=1,2)$ are two solutions to \eqref{SDE} with the initial data $Y_0^{i}$ respectively, then we have the following stability estimate
	\begin{equation}\label{SDEstab1}
	\EE[|Y_t^{1}-Y_t^{2}|^2]\leq \left(1-2t(\lambda_K^-+\lambda_V^-)e^{-2t(\lambda_K^-+\lambda_V^-)} \right)\EE[|Y_0^{1}-Y_0^{2}|^2]\quad \mbox{for all }t\in[0,T]\,,
	\end{equation}
	where $\lambda_K^-$  and $\lambda_V^-$ are  defined in \eqref{lambda}.
\end{thm}
\begin{proof}
	Under Assumption \ref{asum1}, it follows from \cite[Theorem 1.5, Theorem 1.6]{carrillo2014nonlocal} that the PDE \eqref{PDE} has a unique weak solution $\mu_t\in\mathcal{P}_2( D)$. Since $\nabla K$ and $\nabla V$ are continuous and bounded, $F(x_t):=-\nabla K\ast \mu_t(x_t)-\nabla V(x_t)$ is also continuous and bounded. According to  \cite[Lemma 2.4]{carrillo2014nonlocal}, there exists $Y_t$ satisfying the following ODE
	\begin{equation*}
	\rd Y_t=P_{Y_t}\left(F(Y_t)\right)\dt=P_{Y_t}\left(-\nabla K\ast \mu_t(Y_t)-\nabla V(Y_t)\right)\dt.
	\end{equation*}
	Following a similar argument as in Proposition \ref{pdetosde}, one can show that $Law(Y_t)=\mu_t$. Hence we have the existence to \eqref{SDE}.
	
	Assume $Y_t^{i}$, $(i=1,2)$ are two solutions to \eqref{SDE} with the initial data $Y_0^{i}$ respectively. Applying It\^{o}'s formula, it yields that
	\begin{align}
	|Y_t^1-Y_t^2|^2&=|Y_0^1-Y_0^2|^2+2\int_0^t  \la Y_s^1-Y_s^2, P_{Y_s^1}\left(F(Y_s^1)\right)-P_{Y_s^2}\left(F(Y_s^2)\right)\ra\ds\notag\\
	&=|Y_0^1-Y_0^2|^2+2\int_0^t  \la Y_s^1-Y_s^2, F(Y_s^1)-F(Y_s^2)\ra\ds\notag\\
	&\quad +2\int_0^t  \la Y_s^1-Y_s^2, P_{Y_s^1}\left(F(Y_s^1)\right)-F(Y_s^1)+F(Y_s^2)-P_{Y_s^2}\left(F(Y_s^2)\right)\ra\ds
	\end{align}
By the $\lambda$-convexities of $K$ and $V$, one has
	\begin{equation}
	\EE\left[\la Y_s^1-Y_s^2, F(Y_s^1)-F(Y_s^2)\ra\right]\leq -2(\lambda_K^-+\lambda_V^-)\EE\left[|Y_s^1-Y_s^2|^2\right] \,.
	\end{equation}
Moreover, it follows from the convexity of the domain $D$ that
	\begin{align}
	&\la Y_s^1-Y_s^2, P_{Y_s^1}\left(F(Y_s^1)\right)-F(Y_s^1)+F(Y_s^2)-P_{Y_s^2}\left(F(Y_s^2)\right)\ra\notag\\
	=& \la Y_s^1-Y_s^2, P_{Y_s^1}\left(F(Y_s^1)\right)-F(Y_s^1)\ra +\la Y_s^1-Y_s^2, F(Y_s^2)-P_{Y_s^2}\left(F(Y_s^2)\right)\ra\leq 0.
	\end{align}
	Here $P_{Y_s^1}\left(F(Y_s^1)\right)-F(Y_s^1)$ is in the inward normal direction to $\partial D$ at $Y_s^1$.
	Thus one has
	\begin{equation}
	\EE\left[|Y_t^1-Y_t^2|^2\right]\leq \EE\left[|Y_0^1-Y_0^2|^2\right]-2(\lambda_K^-+\lambda_V^-)\int_0^t\EE\left[|Y_s^1-Y_s^2|^2\right]\ds\,,
	\end{equation}
	which leads to \eqref{SDEstab1} by the Gronwall's inequality.
\end{proof}

\subsection{Zero-diffusion limit}
Now we can obtain an explicit convergence rate between $\mu_t^\nu$ and $\mu_t$ (weak solutions to \eqref{PDEnu} and \eqref{PDE} respectively) by using the coupling method, where we  treat $\mu_t^\nu$ and $\mu_t$ as distributions of $Y_t^\nu$ and $Y_t$ respectively and  use the fact that
\begin{equation}
\mathcal{W}_2^2(\mu_t^\nu,\mu_t)\leq  \EE\left[|Y_t^\nu-Y_t|^2\right].
\end{equation}

Our main theorem regarding the convergence of weak solutions (in the Wasserstein metric) states:
\begin{thm}\label{mainthm1}
	Assume $K$, $V$ and $D$ satisfy Assumption \ref{asum1}.  For any $T>0$, let $\mu_t^\nu$ and $\mu_t$ be weak solutions to \eqref{PDEnu} and \eqref{PDE} on $[0, T]$ respectively. Then it holds that
	\begin{equation*}
	\mathcal{W}_2^2(\mu_t^\nu,\mu_t)\leq 2d\nu t\left(1-2(\lambda_K^-+\lambda_V^-)te^{-2(\lambda_K^-+\lambda_V^-)t}\right)\quad \mbox{for all }t\in[0,T]\,,
	\end{equation*}
	where $\mathcal{W}_2$ denotes the $2$-Wasserstein metric as in \eqref{wass} and $\lambda_K^-$  and $\lambda_V^-$ are  defined in \eqref{lambda}.
\end{thm}
\begin{proof}
	Suppose that $(Y_t^\nu,\tilde R_t^\nu)$ and $Y_t$ satisfy \eqref{SDEnu} and \eqref{SDE} respectively. Then it follows from  It\^{o}'s formula that
	\begin{align}\label{ito}
	|Y_t^\nu-Y_t|^2&=2\int_0^t\viint{Y_s^\nu-Y_s}{-\nabla K\ast\mu_s^\nu(Y_s^\nu)-\nabla V(Y_s^\nu)-P_{Y_s}\left(-\nabla K\ast \mu_s(Y_s)-\nabla V(Y_s)\right)}\ds\notag\\
	&\quad+2\sqrt{2\nu}\viint{Y_s^\nu-Y_s}{{\rm d} B_s}-2\int_0^t\viint{Y_s^\nu-Y_s}{n(Y_s^\nu)}{\rm d}|\tilde{R}^\nu|_s+2d\nu t\notag\\
	&\leq  2\int_0^t\viint{Y_s^\nu-Y_s}{-\nabla K\ast\mu_s^\nu(Y_s^\nu)-\nabla V(Y_s^\nu)-P_{Y_s}\left(-\nabla K\ast \mu_s(Y_s)-\nabla V(Y_s)\right)}\ds\notag\\
	&\quad+2\sqrt{2\nu}\int_0^t\viint{Y_s^\nu-Y_s}{{\rm d} {B}_s}+2d\nu t,
	\end{align}
	where in the second inequality we have used the the convexity of the domain $D$. Denote that
	\begin{equation}
	\xi_s^\nu(Y_s^\nu):=-\nabla K\ast\mu_s^\nu(Y_s^\nu)-\nabla V(Y_s^\nu)\,,\quad
	\xi_s(Y_s):=-\nabla K\ast \mu_s(Y_s)-\nabla V(Y_s)\,.
	\end{equation}
By the convexity of the domain $D$, we have 	
\begin{equation}
   \viint{Y_s^\nu-Y_s}{\xi_s(Y_s)-P_{Y_s}\left(\xi_s(Y_s)\right)}\leq 0 \,.
\end{equation}
Thus
\begin{align} \label{bound:Y-nu-Y-bdry}
	&\viint{Y_s^\nu-Y_s}{\xi_s^\nu(Y_s^\nu)-P_{Y_s}\left(\xi_s(Y_s)\right)}\notag\\
	=&\viint{Y_s^\nu-Y_s}{\xi_s^\nu(Y_s^\nu)-\xi_s(Y_s)}+\viint{Y_s^\nu-Y_s}{\xi_s(Y_s)-P_{Y_s}\left(\xi_s(Y_s)\right)}\notag\\
	\leq&\viint{Y_s^\nu-Y_s}{\xi_s^\nu(Y_s^\nu)-\xi_s(Y_s)},
	\end{align}
Applying~\eqref{bound:Y-nu-Y-bdry} and the $\lambda_V$ convexity of $V$ in~\eqref{ito}, we obtain the bound
	\begin{align}\label{aftercon}
	|Y_t^\nu-Y_t|^2&\leq 2\int_0^t\viint{Y_s^\nu-Y_s}{-\nabla K \ast\mu_s^\nu(Y_s^\nu)+\nabla K\ast\mu_s(Y_s)}\ds+2\int_0^t\viint{Y_s^\nu-Y_s}{-\nabla V(Y_s^\nu)+\nabla V(Y_s)}\ds\notag\\
	&\quad+2\sqrt{2\nu}\int_0^t\viint{Y_s^\nu-Y_s}{{\rm d}{B}_s}+2d\nu t \notag\\
	&\leq 2\int_0^t\viint{Y_s^\nu-Y_s}{-\nabla K \ast\mu_s^\nu(Y_s^\nu)+\nabla K\ast\mu_s(Y_s)}\ds-2\lambda_V^-\int_0^t |Y_s^\nu-Y_s|^2\ds\notag\\
	&\quad+2\sqrt{2\nu}\int_0^t\viint{Y_s^\nu-Y_s}{{\rm d}{B}_s}+2d\nu t \,.
	\end{align}
Next we take expectation on \eqref{aftercon} and use Lemma \ref{lmconvex} to obtain that
	\begin{align}
	\EE\left[|Y_t^\nu-Y_t|^2\right]\leq -2(\lambda_K^-+\lambda_V^-)\int_0^t\EE[|Y_s^\nu-Y_s|^2]\ds+2\sqrt{2\nu}\EE\left[\int_0^t\viint{Y_s^\nu-Y_s}{{\rm d}{B}_s}\right]+2d\nu t.
	\end{align}
	Since the domain $D$ is bounded, it holds that
	\begin{equation}\label{verify}
	\EE\left[|Y_t^{\nu}-Y_t|^2\right]<+\infty \,.
	\end{equation}
Then according to  \cite[pp.28, Theorem 1]{gihman1979stochastic} one has
	\begin{equation}\label{EB0}
	\EE\left[\int_0^t\viint{Y_s^\nu-Y_s}{{\rm d}{B}_s}\right]=0\,.
	\end{equation}		
Therefore,
	\begin{align}
	\EE\left[|Y_t^\nu-Y_t|^2\right] \leq -2(\lambda_K^-+\lambda_V^-)\int_0^t\EE[|Y_s^\nu-Y_s|^2]\ds+2d\nu t\,.
	\end{align}
An application of Gronwall's inequality then yields that
	\begin{equation}
	\mathcal{W}_2^2(\mu_t^\nu,\mu_t) \leq \EE\left[|Y_t^\nu-Y_t|^2\right] \leq 2d \nu t\left(1-2(\lambda_K^-+\lambda_V^-)te^{-2(\lambda_K^-+\lambda_V^-)t}\right)\quad \mbox{for all }t\in[0,T]\,,
	\end{equation}
	which completes the proof.
\end{proof}

\section{Random particle method}
\label{sect:particle}
In this section we investigate the random particle method for approximating weak solutions of PDEs \eqref{PDEnu} and \eqref{PDE} respectively. To prove the convergence of the random particle method, we  introduce an 
auxiliary  stochastic mean-field dynamics $\{Y_t^i\}_{i=1}^N$ satisfying
\begin{equation}\label{SDEnusystem}
\left\{\begin{array}{l}
Y^{\nu,i}_t=Y_0-\int_{0}^t\left(\nabla K\ast\mu_s^\nu(Y^{\nu,i}_s) +\nabla V(Y^{\nu,i}_s)\right)\ds+\sqrt{2\nu}{B}_t^{\nu,i}-\tilde R_t^{\nu,i}\,, \quad t\geq0\,,\quad i=1,\cdots,N,\\
\tilde R_t^{\nu,i}=\int_0^tn(Y^{\nu,i}_s) \, {\rm d}|\tilde R^{\nu,i}|_s\,,\quad |\tilde R^{\nu,i}|_t=\int_0^t\textbf{1}_{\partial D}(Y^{\nu,i}_s) \, {\rm d}|\tilde R^{\nu,i}|_s\,,\\
\end{array}\right.
\end{equation}
where $(\mu_t^\nu)_{t\geq0}$ is the common distribution  of $\{({Y}_t^{\nu,i})_{t\geq 0}\}_{i=1}^N$  and $\{({B}_t^{\nu,i})_{t\geq 0}\}_{i=1}^N$ are $N$ independent Brownian motions .  Here we set the initial data $(Y_0^{\nu,i})_{i=1,\cdots, N}=(X_0^{\nu,i})_{i=1,\cdots, N}$ i.i.d. with the same distribution $\mu_0(x)$. Therefore $\{({Y}_t^{\nu,i})_{t\geq 0}\}_{i=1}^N$  are $N$ independent copies of the strong solutions to \eqref{SDEnu}. 

The following law of large numbers plays an important role in proving the convergence of the particle method:
\begin{lem}\label{lemLLN}(Law of Large Numbers) Let $\{({Y}_t^{\nu,i})_{t\geq 0}\}_{i=1}^N$ be the i.i.d. processes satisfying the SDE system \eqref{SDEnusystem} with the common distribution $\mu_t^\nu$. Then for all $t \geq 0$, we have
	\begin{equation} \label{convg:LLN}
	\EE\left[\sup\limits_{i=1,\cdots,N}\left|-\frac{1}{N-1}\sum\limits_{j\neq i}^N\nabla K(Y_t^{\nu,i}-Y_t^{\nu,j})+\int_D\nabla K(Y_t^{\nu,i}-y) \rd \mu_t^\nu(y)\right|^2\right] \leq \frac{C}{N-1}\,,
	\end{equation}
where $C$ only depends on $\norm{K}_{C^1(D)}$. 
\end{lem}
\begin{proof}
Fix $i=1$ and denote
		\begin{equation*}
	Z_j:=-\nabla K(Y_t^{\nu,1}-Y_t^{\nu,j})+\int_D\nabla K(Y_t^{\nu,1}-y) \rd \mu_t^\nu(y) \,,
	\qquad j \geq 2 \,.
	\end{equation*} 
Then
	\begin{align}
	-\frac{1}{N-1}\sum\limits_{j\neq 1}^N\nabla K(Y_t^{\nu,1}-Y_t^{\nu,j})+\int_D\nabla K(Y_t^{\nu,1}-y) \rd \mu_t^\nu(y)= \frac{1}{N-1}\sum\limits_{j\neq 1}^NZ_j\,.
	\end{align}
We claim that
	\begin{equation*}
	\EE[Z_jZ_k]=0 \,,
	\qquad j \neq k \,.
	\end{equation*}
Indeed, let $Y_t^{\nu,1}$ be given and denote 
	\begin{equation*}
	\EE'[\cdot]=\EE[\cdot| Y_t^{\nu,1}]\,.
	\end{equation*}
	Since $Y_t^{\nu,j}$ and $Y_t^{\nu,k}$ are independent for $j \neq k$, it holds that ${\EE}'[Z_jZ_k]=\EE'[Z_j]\EE'[Z_k]=0$ when $j\neq  k$. Hence
	\begin{equation*}
	\EE[Z_jZ_k]={\EE}_1{\EE}'[Z_jZ_k]=0\,,
	\qquad j \neq k \, ,
	\end{equation*}
	where ${\EE}_1$ means taking expectation on $Y_t^{\nu,1}$.
	
	Hence, one concludes that
	\begin{align}\label{I_2}
	&\EE\left[\left|-\frac{1}{N-1}\sum\limits_{j\neq 1}^N\nabla K(Y_t^{\nu,1}-Y_t^{\nu,j})+\int_D\nabla K(Y_t^{\nu,1}-y) \rd \mu_t^\nu(y)\right|^2\right] \notag\\
	\leq&\frac{C}{N-1}\EE[|Z_2|^2] 
	\leq \frac{C}{N-1}\EE\left[\left(\nabla K(Y_t^{\nu,2}-Y_t^{\nu,j})\right)^2+\left(\int_D\nabla K(Y_t^{\nu,2}-y) \rd \mu_t^\nu(y)\right)^2\right]
	\leq \frac{C}{N-1}\,,
	\end{align}
	where $C$ only depends on $\norm{K}_{C^1(D)}$.  Estimate \eqref{I_2} holds for $i=2,\cdots, 3$ as well, hence~\eqref{convg:LLN} holds.
\end{proof}

\begin{thm}\label{thmmean}
	Suppose that $D$, $K$ and $V$ satisfy Assumption \ref{asum1}. For any $T>0$, let $\mu_t^\nu \in AC^2_{loc}([0,T];\mathcal{P}_2(D))$ be the weak solution to the equation \eqref{PDEnu} with initial data $\mu_0\in \mathcal{P}_2(D)$ up to time $T$. Furthermore, assume that $\{X_t^{\nu,i}\}_{i=1}^N$ satisfy the particle system \eqref{Rparticlenu} with i.i.d. initial data $\{X_0^{\nu,i}\}_{i=1}^N$ sharing the common law $\mu_0$. Then for any $1\leq p<q<\infty$, there exists a constant $C>0$ depending only on $D$, $\lambda_K$, $\lambda_V$, $\norm{K}_{C^1(D)}$, $q$ and $T$ such that
	\begin{align}\label{propogation}
	&\sup\limits_{t\in[0,T]}\EE\left[\mathcal{W}_p(\munuX,\mu_t^\nu)\right]\leq CN^{-\frac{1}{4}}+C\left\{\begin{aligned}
	&N^{-\frac{1}{2p}}+N^{-\frac{q-p}{qp}}, ~&&\mbox{if }2p>d\mbox{ and }q\neq 2p\,,\\
	&N^{-\frac{1}{2p}}\log(N+1)^{\frac{1}{p}}+N^{-\frac{q-p}{qp}},~ &&\mbox{if }2p=d\mbox{ and }q\neq 2p\,,\\
	&N^{-\frac{1}{d}}+N^{-\frac{q-p}{qp}}, ~ &&\mbox{if }2p<d\mbox{ and }q\neq \frac{d}{d-p}\,,\\
	\end{aligned}\right.
	\end{align}
	where $\munuX:=\frac{1}{N}\sum\limits_{i=1}^N\delta_{X_t^{\nu,i}}$
	is the empirical measure associated to the particle system \eqref{Rparticlenu}.
\end{thm}
\begin{proof}
	First, we show that the $N$ interacting processes $\{X_t^{\nu,i}\}_{i=1}^N$ of \eqref{Rparticlenu} well approximate the processes $\{Y_t^{\nu,i}\}_{i=1}^N$ of \eqref{SDEnusystem}  as $N\rightarrow \infty$. It follows from It\^{o}'s formula that
	\begin{align}\label{Ito3}
	|X_t^{\nu,i}-Y_t^{\nu,i}|^2&=2\int_0^t\left\la X_s^{\nu,i}-Y_s^{\nu,i}, -\frac{1}{N-1}\sum\limits_{j\neq i}^N\nabla K(X_s^{\nu,i}-X_s^{\nu,j})+\int_D\nabla K(Y_s^{\nu,i}-y) \rd \mu_s^\nu(y)\right\ra \ds\notag\\
	&\quad+2\int_0^t\left\la X_s^{\nu,i}-Y_s^{\nu,i}, -\nabla V(X_s^{\nu,i})+\nabla V(Y_s^{\nu,i})\right\ra \ds\notag\\
	&\quad-2\int_0^t\la X_s^{\nu,i}-Y_s^{\nu,i}, n(X_s^{\nu,i})\ra {\, \rm d}|R^{\nu,i}|_s-2\int_0^t\la Y_s^{\nu,i}-X_s^{\nu,i}, n(Y_s^{\nu,i})\ra {\, \rm d}|\tilde R^{\nu,i}|_s \notag \\
	&\leq2\int_0^t\left\la X_s^{\nu,i}-Y_s^{\nu,i}, -\frac{1}{N-1}\sum\limits_{j\neq i}^N\nabla K(X_s^{\nu,i}-X_s^{\nu,j})+\int_D\nabla K(Y_s^{\nu,i}-y) \rd \mu_s^\nu(y)\right\ra \ds\notag\\
	&\quad-2\lambda_V^-\int_0^t |X_s^{\nu,i}-Y_s^{\nu,i}|^2 \ds \,,
	\end{align}
	where in the second inequality we have used the convexity of the domain $D$ and the $\lambda_V$ convexity of $V$. 
	Next we compute that
	\begin{align}
	&\int_0^t\left\la X_s^{\nu,i}-Y_s^{\nu,i}, -\frac{1}{N-1}\sum\limits_{j\neq i}^N\nabla K(X_s^{\nu,i}-X_s^{\nu,j})+\int_D\nabla K(Y_s^{\nu,i}-y) \rd \mu_s^\nu(y)\right\ra \ds \notag \\
	=&\frac{1}{N-1}\sum\limits_{j\neq i}^N\int_0^t\left\la X_s^{\nu,i}-Y_s^{\nu,i}, -\nabla K(X_s^{\nu,i}-X_s^{\nu,j})+\nabla K(Y_s^{\nu,i}-Y_s^{\nu,j})\right\ra \ds\notag\\
	&+\int_0^t\left\la X_s^{\nu,i}-Y_s^{\nu,i}, -\frac{1}{N-1}\sum\limits_{j\neq i}^N\nabla K(Y_s^{\nu,i}-Y_s^{\nu,j})+\int_D\nabla K(Y_s^{\nu,i}-y) \rd \mu_s^\nu(y)\right\ra \ds \notag\\
	=:& \frac{1}{N-1}\sum\limits_{j\neq i}^N \int_0^tI_s^{ij}\ds+C\int_0^t|\mathcal{H}_s^i|\ds\,,
	\end{align}	
	where $C$ only depends on $D$ and
	\begin{equation}
	\mathcal{H}_s^i =-\frac{1}{N-1}\sum\limits_{j\neq i}^N\nabla K(Y_s^{\nu,i}-Y_s^{\nu,j})+\int_D\nabla K(Y_s^{\nu,i}-y) \rd \mu_s^\nu(y).
	\end{equation}
By \eqref{lamcon}, 
	\begin{align}\label{sum}
	&\sum\limits_{ \scriptstyle  i, j =1 \atop \scriptstyle j \neq i}^N I_s^{ij}=\sum\limits_{ \scriptstyle  i, j =1 \atop \scriptstyle j \neq i}^N\viint{X_s^{\nu,i}-Y_s^{\nu,i}}{ -\nabla K ({X}_s^{\nu,i}-{X}_s^{\nu,j})+\nabla K (Y_s^{\nu,i}-Y_s^{\nu,j})}\notag\\
	=&\frac{1}{2}\sum\limits_{ \scriptstyle  i, j =1 \atop \scriptstyle j \neq i}^N\viint{X_s^{\nu,i}-Y_s^{\nu,i}-X_s^{\nu,j}+Y_s^{\nu,j}}{ -\nabla K ({X}_s^{\nu,i}-{X}_s^{\nu,j})+\nabla K (Y_s^{\nu,i}-Y_s^{\nu,j})}\notag\\
	\leq&-\frac{\lambda_K}{2}\sum\limits_{ \scriptstyle  i, j =1 \atop \scriptstyle j \neq i}^N|X_s^{\nu,i}-Y_s^{\nu,i}-X_s^{\nu,j}+Y_s^{\nu,j}|^2
	\leq -\lambda_K^-N(N-1) \sup\limits_{i=1,\cdots,N}|X_s^{\nu,i}-Y_s^{\nu,i}|^2,
	\end{align}
Summing all of \eqref{Ito3} and applying \eqref{sum}, we have
	\begin{align}
	\sum\limits_{i=1}^N 	|X_t^{\nu,i}-Y_t^{\nu,i}|^2
	\leq -2(\lambda_K^-+\lambda_V^-)N\int_0^t \sup\limits_{i=1,\cdots,N}|X_s^{\nu,i}-Y_s^{\nu,i}|^2\ds+2 CN\int_0^t\sup\limits_{i=1,\cdots,N} |\mathcal{H}_s^i|\ds\,.
	\end{align}
By the exchangeability of $\{X_t^{\nu,i}-Y_t^{\nu,i}\}_{i=1}^N$, if we take the expectation of the above inequality, then
	\begin{align}\label{before}
	\EE\left[|X_t^{\nu,i}-Y_t^{\nu,i}|^2\right]  \leq -2(\lambda_K^-+\lambda_V^-)\int_0^t\EE \left[\sup\limits_{i=1,\cdots,N}|X_s^{\nu,i}-Y_s^{\nu,i}|^2\right]\ds+2 C\int_0^t\EE\left[\sup\limits_{i=1,\cdots,N}| \mathcal{H}_s^i|\right]\ds\,.
	\end{align}
	According to Lemma \ref{lemLLN},  we have
	\begin{equation}
	\EE\left[\sup\limits_{i=1,\cdots,N} |\mathcal{H}_s^i|\right]\leq \frac{C}{\sqrt{N-1}}
\qquad \text{for all $s\in[0,t]$} \,,
	\end{equation}
	where $C$ only depends on $\norm{K}_{C^1(D)}$. 
Together with \eqref{before}, it yields that
	\begin{align}
	\EE\left[|X_t^{\nu,i}-Y_t^{\nu,i}|^2\right]  \leq -2(\lambda_K^-+\lambda_V^-)\int_0^t\EE \left[\sup\limits_{i=1,\cdots,N}|X_s^{\nu,i}-Y_s^{\nu,i}|^2\right]\ds+2 \frac{Ct}{\sqrt{N-1}}\,,
	\end{align}
	where $C$ depends only on $D$ and $\norm{K}_{C^1(D)}$. 
	Applying Gronwall's inequality, we have
	\begin{align}\label{Gron}
	\EE\left[|X_t^{\nu,i}-Y_t^{\nu,i}|^2\right]\leq \frac{Ct}{\sqrt{N-1}}\left(1-2(\lambda_K^-+\lambda_V^-)te^{-2(\lambda_K^-+\lambda_V^-)t}\right)\,,
	\end{align}
	for all $t\in[0,T]$.
	
	Now recall the definition of the empirical measures
	\begin{equation*}
	\munuX=\frac{1}{N}\sum\limits_{i=1}^N\delta_{X_t^{\nu,i}}\,,
	\qquad \mu_t^{\nu,Y}:=\frac{1}{N}\sum\limits_{i=1}^N\delta_{Y_t^{\nu,i}}\,,
	\end{equation*}
	and notice that
	\begin{equation*}
	\pi_\ast=\frac{1}{N}\sum_{i=1}^{N}\delta_{X_t^{\nu,i}}\delta_{Y_t^{\nu,i}}\in \Lambda(\munuX,\mu_t^{\nu,Y})\,.
	\end{equation*}
	Then by the definition of $\mathcal{W}_\infty$ distance one has 
	\begin{equation}\label{winf}
	\mathcal{W}_\infty(\munuX,\mu_t^{\nu,Y})\leq \pi_\ast\mathop{\mbox{-ess sup }}\limits_{(x,y)\in D\times D}|x-y|\leq\sup\limits_{i=1,\cdots,N}|X_t^{\nu,i}-Y_t^{\nu,i}| \,.
	\end{equation}
	Applying \eqref{Gron} leads to
	\begin{equation}\label{X-Y}
	\sup\limits_{t\in[0,T]}\EE\left[\mathcal{W}_\infty(\munuX,\mu_t^{\nu,Y})\right]\leq \EE\left[\sup\limits_{i=1,\cdots,N}|X_t^{\nu,i}-Y_t^{\nu,i}|\right]\leq \frac{C_T}{N^{\frac{1}{4}}}\,,
	\end{equation}
	where $C_T$ depends only on $T$, $D$, $\lambda_K^-$, $\lambda_V^-$ and $\norm{K}_{C^1(D)}$.
	
	Furthermore, the convergence estimate obtained in \cite[Theorem 1]{fournier2015rate} gives
	\begin{equation}\label{fournier}
	\EE\left[\mathcal{W}_p(\mu_t^{\nu,Y},\mu_t^\nu)\right]\leq C\left\{\begin{aligned}
	&N^{-\frac{1}{2p}}+N^{-\frac{q-p}{qp}}, ~&&\mbox{if }2p>d\mbox{ and }q\neq 2p\,,\\
	&N^{-\frac{1}{2p}}\log(N+1)^{\frac{1}{p}}+N^{-\frac{q-p}{qp}},~ &&\mbox{if }2p=d\mbox{ and }q\neq 2p\,,\\
	&N^{-\frac{1}{d}}+N^{-\frac{q-p}{qp}}, ~ &&\mbox{if }2p<d\mbox{ and }q\neq \frac{d}{d-p}\,.\\
	\end{aligned}\right.
	\end{equation}
Combining this with \eqref{X-Y} and the property of the Wasserstein distances \eqref{wrela}, we derive \eqref{propogation} by the triangle inequality such that
	\begin{align*}
	\EE\left[\mathcal{W}_p(\munuX,\mu_t^\nu)\right]&\leq \EE\left[\mathcal{W}_p(\munuX,\mu_t^{\nu,Y})\right]+\EE\left[\mathcal{W}_p(\mu_t^{\nu,Y},\mu_t^\nu)\right]\notag\\
	&\leq\EE\left[\mathcal{W}_\infty(\munuX,\mu_t^{\nu,Y})\right]+\EE\left[\mathcal{W}_p(\mu_t^{\nu,Y},\mu_t^\nu)\right]\,.
	\end{align*}
\end{proof}

It was proved in \cite{carrillo2014nonlocal} that the empirical measure associated to the particle system \eqref{Rparticle} well approximates the weak solution of \eqref{PDE}. The precise statement is
\begin{thm} \label{thmmean1}(\cite[Proposition 3.1]{carrillo2014nonlocal})
	Assume that $D$, $K$ and $V$ satisfy Assumption \ref{asum1}. For any $T>0$, let $\mu_t(x)\in AC^2_{loc}([0,T];\mathcal{P}_2( D))$ be the unique weak solution to the plain aggregation equation \eqref{PDE} with initial data $\mu_0(x)\in\mathcal{P}_2(D)$ up to the time $T$.  Furthermore, assume that $\{X_t^{i}\}_{i=1}^N$ satisfy the particle system \eqref{Rparticle} with i.i.d. initial data $\{X_0^{i}\}_{i=1}^N$ sharing the common law $\mu_0$.
	Then the empirical measure 
	\begin{equation}\label{muX}
	\mu_t^{X}:=\frac{1}{N}\sum\limits_{i=1}^N\delta_{X_t^{i}}\,,
	\end{equation}
associated to the particle system \eqref{Rparticle} satisfies
	\begin{equation}\label{prop1}
	\mathcal{W}_2(\mu_t^{X},\mu_t)\leq \exp\left(-(\lambda_K^-+\lambda_V^-)t\right)\mathcal{W}_2(\mu_0^{X},\mu_0)\,,
	\end{equation}
	for all $t\in[0,T]$, where $\lambda_K^-$  and $\lambda_V^-$ are  defined in \eqref{lambda}. Hence, 
	\begin{equation}
	\mathcal{W}_2(\mu_t^{X},\mu_t)\rightarrow 0 \quad\mbox{as }N\rightarrow \infty \,,
	\end{equation}
since the initial data satisfies~\eqref{LLN-non-diffuse}.
\end{thm}
\begin{rmk}
Since $\{X_0^{i}\}_{i=1}^N$ are i.i.d. and share the the common law $\mu_0$, it follows from  \cite[Theorem 1]{fournier2015rate} that
	\begin{equation} \label{LLN-non-diffuse}
	\EE\left[\mathcal{W}_2(\mu_0^{X},\mu_0)\right]\leq C\left\{\begin{aligned}
	&N^{-\frac{1}{4}}+N^{-\frac{q-2}{2q}}, ~&&\mbox{if }4>d\mbox{ and }q\neq 4\,,\\
	&N^{-\frac{1}{4}}\log(N+1)^{\frac{1}{2}}+N^{-\frac{q-2}{2q}},~ &&\mbox{if }4=d\mbox{ and }q\neq 4\,,\\
	&N^{-\frac{1}{d}}+N^{-\frac{q-2}{2q}}, ~ &&\mbox{if }4<d\mbox{ and }q\neq \frac{d}{d-2}\,.\\
	\end{aligned}\right.
	\end{equation}
 One can now use \eqref{LLN-non-diffuse} in \eqref{prop1} to get an estimate for $\EE \left[ \mathcal{W}_2(\mu_t^{X},\mu_t) \right]$. Note that compared to this estimate, in \eqref{propogation}, the random particle method for aggregation-diffusion model \eqref{PDEnu} has an additional term $N^{-\frac{1}{4}}$. The reason is that with the presence of the diffusion, particles evolve via Brownian motion. The extra term $N^{-\frac{1}{4}}$ results from the compensation of the law of large numbers for the randomness generated by the Brownian motion.
\end{rmk}

Theorem \ref{thmmean} and Theorem \ref{thmmean1} show that the empirical measures $\munuX$ and $\mu_t^{X}$ associated to the particle methods approximate the exact solutions $\mu_t^\nu$ and $\mu_t$ to equations \eqref{PDEnu} and \eqref{PDE}, respectively, as $N\rightarrow \infty$. The following theorem establishes the convergence  of $\mu_t^\nu$ to $\mu_t$ as $\nu \to 0$ (see Theorem \ref{mainthm1}) at the particle level:
\begin{thm}\label{thmparticle}
	Assume that $\{X_0^i\}_{i=1}^N$ are $N$ i.i.d. random variables with the common law $\mu_0$. Let $\{X_t^{\nu,i}\}_{i=1}^N$ and  $\{X_t^{i}\}_{i=1}^N$ satisfy particle systems \eqref{Rparticlenu} and \eqref{Rparticle} respectively with the same initial data $\{X_0^i\}_{i=1}^N$. Furthermore, suppose that 
	$\munuX$ and $\mu_t^{X}$	are the empirical measures associated to  particle systems \eqref{Rparticlenu} and \eqref{Rparticle} respectively.
	Then it holds that
	\begin{align}
	\EE\left[\mathcal{W}_\infty^2(\munuX,\mu_t^X)\right]\leq 2d\nu t\left(1-2(\lambda_K^-+\lambda_V^-)te^{-2(\lambda_K^-+\lambda_V^-)t}\right)\quad\mbox{ for all }t\in[0,T]\,,
	\end{align} 
	where $\lambda_K^-$  and $\lambda_V^-$ are  defined in \eqref{lambda}.
\end{thm}
\begin{proof}
	For $i=1,\cdots,N$, one applies  It\^{o}'s formula and gets
	\begin{align}\label{eqi}
	&|X_t^{\nu,i}-X_t^i|^2\notag \\
	=&\frac{2}{N-1}\sum\limits_{j\neq i}^N\int_0^t\viint{X_s^{\nu,i}-X_s^i}{ -\nabla K ({X}_s^{\nu,i}-{X}_s^{\nu,j}) -\nabla V({X}_s^{\nu,i})- P_{X_s^i}\left(-\nabla K({X}_s^{i}-{X}_s^{j}) -\nabla V({X}_s^{i}) \right)}\ds\notag\\
	&+2\sqrt{2\nu}\viint{X_s^{\nu,i}-X_s^i}{{\rm d}{B}_s^{\nu,i}}-2\int_0^t\viint{X_s^{\nu,i}-X_s^i}{n(X_s^{\nu,i})}{\rm d}|\tilde{R}^{\nu,i}|_s+2d\nu t\notag\\
	\leq  &\frac{2}{N-1}\sum\limits_{j\neq i}^N\int_0^t\viint{X_s^{\nu,i}-X_s^i}{ -\nabla K({X}_s^{\nu,i}-{X}_s^{\nu,j})+ \nabla K({X}_s^{i}-{X}_s^{j})}\ds\notag\\
	&+2\int_0^t\viint{X_s^{\nu,i}-X_s^i}{ -\nabla V({X}_s^{\nu,i})+\nabla V({X}_s^{i})}\ds
	+2\sqrt{2\nu}\int_0^t\viint{X_s^{\nu,i}-X_s^i}{{\rm d}{B}_s^{\nu,i}}+2d\nu t \notag\\
	\leq &\frac{2}{N-1}\sum\limits_{j\neq i}^N\int_0^t\viint{X_s^{\nu,i}-X_s^i}{ -\nabla K({X}_s^{\nu,i}-{X}_s^{\nu,j})+ \nabla K({X}_s^{i}-{X}_s^{j})}\ds\notag\\
	&-2\lambda_V^-\int_0^t|X_s^{\nu,i}-X_s^i|^2\ds+2\sqrt{2\nu}\int_0^t\viint{X_s^{\nu,i}-X_s^i}{{\rm d}{B}_s^{\nu,i}}+2d\nu t\,,
	\end{align}
	where we have used the convexity of the domain and the $\lambda_V$-convexity of $V$.
	Notice that by~\eqref{lamcon},
	\begin{align}
	&\sum\limits_{ \scriptstyle  i, j =1 \atop \scriptstyle j \neq i}^N\viint{X_s^{\nu,i}-X_s^i}{ -\nabla K ({X}_s^{\nu,i}-{X}_s^{\nu,j})+\nabla K ({X}_s^{i}-{X}_s^{j})}\notag\\
	=&\frac{1}{2}\sum\limits_{ \scriptstyle  i, j =1 \atop \scriptstyle j \neq i}^N\viint{X_s^{\nu,i}-X_s^i-X_s^{\nu,j}+X_s^j}{ -\nabla K ({X}_s^{\nu,i}-{X}_s^{\nu,j})+\nabla K ({X}_s^{i}-{X}_s^{j})}\notag\\
	\leq&-\frac{\lambda_K}{2}\sum\limits_{ \scriptstyle  i, j =1 \atop \scriptstyle j \neq i}^N|X_s^{\nu,i}-X_s^i-X_s^{\nu,j}+X_s^j|^2
	\leq -\lambda_K^-N(N-1) \sup\limits_{i=1,\cdots,N}|X_s^{\nu,i}-X_s^i|^2\,,
	\end{align}
Summation over $i$ in \eqref{eqi} gives that
	\begin{align}
	\sum\limits_{i=1}^N|X_t^{\nu,i}-X_t^i|^2&\leq  -2\lambda_K^-N \int_0^t\sup\limits_{i=1,\cdots,N}|X_s^{\nu,i}-X_s^i|^2\ds-2\lambda_V^-\sum\limits_{i=1}^N\int_0^t|X_s^{\nu,i}-X_s^i|^2\ds\notag\\
	&\quad+2\sqrt{2\nu}\sum\limits_{i=1}^N\int_0^t\viint{X_s^{\nu,i}-X_s^i}{{\rm d}{B}_s^{\nu,i}}+2Nd\nu t\,,
	\end{align}
Taking expectation on the above inequality and using the exchangeability of $\{X_t^{\nu,i}-X_t^i\}_{i=1}^N$, one has
	\begin{equation}
	\EE[|X_t^{\nu,i}-X_t^i|^2]\leq -2(\lambda_K^-+\lambda_V^-)\int_0^t\EE\left[\sup\limits_{i=1,\cdots,N}|X_s^{\nu,i}-X_s^i|^2\right]\ds+2d\nu t\,.
	\end{equation}
	Applying Gronwall's inequality, we have
	\begin{equation}
	\EE\left[\sup\limits_{i=1,\cdots,N}|X_t^{\nu,i}-X_t^i|^2\right]\leq 2d\nu t\left(1-2(\lambda_K^-+\lambda_V^-)te^{-2(\lambda_K^-+\lambda_V^-)t}\right)
	\end{equation}
	for any $t\in[0,T]$. This leads to
	\begin{align*}
	\EE\left[\mathcal{W}_\infty^2(\munuX,\mu_t^X)\right]\leq \EE\left[\sup\limits_{i=1,\cdots,N}|X_t^{\nu,i}-X_t^i|^2\right]\leq 2d\nu t\left(1-2(\lambda_K^-+\lambda_V^-)te^{-2(\lambda_K^-+\lambda_V^-)t}\right)
	\end{align*}
	by using \eqref{winf}.
\end{proof}

\section{Numerical Results}
\label{sect:numerics}
In this section we numerically verify the convergence in Theorem \ref{thmparticle} using representative swarming models in subsets of $\mathbb{R}^2$. The purpose of this study is to exhibit the theoretical $\mathcal{O}(\nu)$ asymptotic convergence, improved from $\mathcal{O}(\nu^{\frac{2}{d+2}})$ in \cite{zhang2017continuity}, for interaction potentials of reduced regularity. We also show that for a small fixed $\nu>0$, the stochastic and deterministic particle systems diverge in time according to an exponential growth function that matches that of the growth bound in Theorem \ref{thmparticle}. 

All simulations in this section use Lagrangian particle-tracking with an explicit-Euler scheme suitably modified to account for boundary conditions. Computations were done in MATLAB, utilizing the computing resources at the Compute Canada's Cedar cluster. In what follows, numerical approximations will be represented with tildes, i.e. $\npartdw{n} = \left\{\npartd{n}{i}\right\}_{i=1}^N$ and $\npart{n}{} = \left\{\npart{n}{i}\right\}_{i=1}^N$ denote numerical solutions to the stochastic and deterministic particle systems \eqref{Rparticlenu} and \eqref{Rparticle}, with empirical measures $\nempdif{n}{}$ and $\nemp{n}{}$, respectively. Also, $\left\{\tau_n\right\}_{n=1}^L$ with $0=\tau_0 < \tau_1 < \cdots < \tau_L = T$ represents a partition of the interval $[0,T]$ with $\Delta t_n := \tau_{n+1}-\tau_n$. 

\subsection{Particle Method}
\begin{figure}
        \includegraphics[trim={28 120 55 100},clip,width=0.6\textwidth]{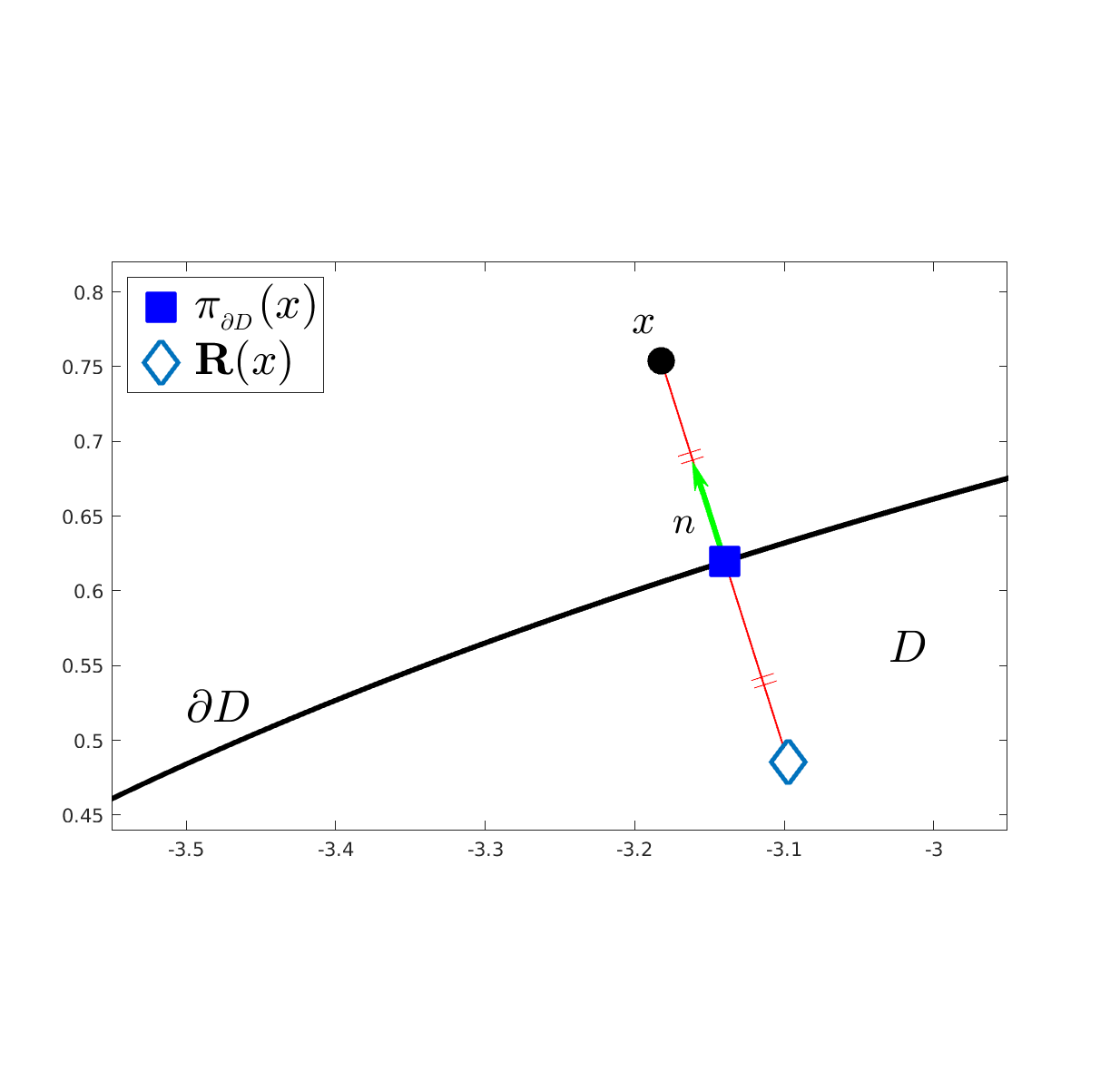}
\caption{Illustration of the projection $\bfP$ (see \eqref{eqn:proj-bdry}) and the reflection operator $\bfR$ (see \eqref{eqn:reflection-op}) for a point $x \notin D$.}
\label{bc_plot}
\end{figure}
We consider systems of $N$ particles with equal mass $1/N$. For numerical solutions of the stochastic particle system \eqref{Rparticlenu} we use the symmetrized reflection scheme from \cite{bossy2004symmetrized}. The method was shown to produce a stochastic process which converges in law to the solution of an associated partial differential equation with Neumann boundary conditions. 

For this purpose, we define the reflection operator $\bfR$:
\begin{align}
\label{eqn:reflection-op}
\bfR(x)=\begin{cases}
x, \quad&\mbox{if }x \in D,\\
x- 2 [(x-\bfP x) \cdot n ]\, n,\quad &\mbox{otherwise},
\end{cases}
\end{align}
where $\bfP x \in \partial D$ denotes the closest point on the boundary to $x \notin D$, i.e.,
\begin{equation}
\label{eqn:proj-bdry}
\vert x - \bfP x \vert= \text{min}_{z\in\partial D} \vert x-z \vert,
\end{equation}
and $n$ is the outward unit normal to $\partial D$ at $\bfP x$. Note that for a given $x \notin D$, $ \bfP x$ is uniquely determined due to the convexity of $D$. 

In words, the operator $\bfR$ takes the mirror reflection across $\partial D$ of points outside $D$ (see Figure \ref{bc_plot}). For points in small enough neighbourhoods of $D$, such reflections lie in $D$. The symmetrized Euler scheme from \cite{bossy2004symmetrized} is simply the stochastic Euler method with symmetric reflection upon crossing the boundary, which for the particle system \eqref{Rparticlenu} it amounts to:
\begin{equation} \label{eq:stochpart}
 \npartd{n+1}{i} = \bfR \biggl(  \npartd{n}{i}-\Delta t_n \Bigl(\frac{1}{N-1}\sum_{k \neq i} \nabla K (\npartd{n}{i}-\npartd{n}{k})+\nabla V(\npartd{n}{i})\Bigr)+\sqrt{2\nu \Delta t_n} N^d(0,1) \biggr).
\end{equation}

To compute numerical solutions of the non-diffusive particle system \eqref{Rparticle} we use an Euler time-stepping, and if the point exits $D$ at the end of the time step, we project the endpoint on $\partial D$ using \eqref{eqn:proj-bdry}. Hence, by defining $\pi_{_{\partial D}}$ to be the identity map for $ x \in D$, and given by the projection \eqref{eqn:proj-bdry} for $ x \notin D$, the update at time $\tau_{n+1}$ for system \eqref{Rparticle} is given by:
 \begin{equation}\label{eq:detpart}
 \npart{n+1}{i} = \pi_{_{\partial D}}\biggl( \npart{n}{i}-\Delta t_n\Bigl(\frac{1}{N-1}\sum_{k \neq i} \nabla K (\npart{n}{i}-\npart{n}{k})+\nabla V(\npart{n}{i})\Bigr)\biggr).
 \end{equation}

\subsection{Wasserstein Distance Between Empirical Measures}
\label{subsect:num-wasserstein}
From the numerical particle systems $\npartdw{n}$ and $\npart{n}{}$, we compute $\mathcal{W}^2_\infty\left(\nempdif{n},\nemp{n}\right)$ by utilizing a convenient reduction of the Monge-Kantorovich transport problem in the case of empirical measures \cite{villani2008optimal}:
\begin{equation}\label{eqn:Wdist}
\mathcal{W}^2_\infty\left(\nempdif{n},\nemp{n}\right) = \underset{\sigma\in S_N}{\inf}\left\{\sup_i\left\vert \npartd{n}{i}-\npart{n}{\sigma(i)}\right\vert^2\right\},
\end{equation}
where $S_N$ is the set of permutations on $N$ elements. Hence, we verify Theorem \ref{thmparticle} by using a simple matching algorithm to compute \eqref{eqn:Wdist}. We illustrate the procedure below.

Let $\npartdw{n}$ and $\npart{n}{}$ represent the $2N$ vertices of a complete, weighted bipartite graph $G$ with bipartite adjacency matrix $A$ and edge weights $A_{ij} =\left\vert \npartd{n}{i}-\npart{n}{j}\right\vert^2$. Finding an optimal transport plan between the associated empirical measures $\nempdif{n}$ and $\nemp{n}$ is equivalent to finding a perfect matching in $G$ with least maximum-matching cost, where a perfect matching is defined to be a collection of edges from $G$ which is pair-wise non-adjacent and visits each vertex exactly once. We do this as follows. For convenience, let $C$ be the increasing sequence of unique edge costs from $A$:  
\[C = \left\{c_k \, \Big\vert \, c_k = A_{ij} \text{  for some  } (i,j), \text{  } c_k<c_{k+1} \text{ for }  0\leq k\leq N^2-1\right\}.\]
For each $k \leq \vert C\vert$, let $B^k$ be the bipartite adjacency matrix for the subgraph of $G$ obtained by removing all edges of cost greater than $c_k$:
\[B^k_{ij} := \begin{cases} 1, & A_{ij} \leq c_k \\ 0, &A_{ij}>c_k.\end{cases}\] 
Then, 
\[\mathcal{W}^2_\infty\left(\nempdif{n},\nemp{n}\right) = \min\left\{c_k \, \Big\vert\,  B^k \text{ contains a perfect matching }\right\}.\]
Thus, to compute \eqref{eqn:Wdist}, we start with $k_0$, where $k_0$ is the smallest integer such that $B^{k_0}$ represents a spanning subgraph of $G$ (this is necessary to ensure that each particle is represented in the subgraph, and hence that a transport plan exists), then iterate through $k$ until $B^k$ contains a perfect matching. This requires that we calculate the length of the maximal matchings of $B^k$ at each $k$, for which we use MATLAB's native \mcode{dmperm} command to perform a Dulmage-Mendelsohn decomposition of $B^k$. Letting $\overline{k}$ denote the first $k$ such that $B^k$ contains a perfect matching, $B^{\overline{k}}$ then contains a one-to-one matching between $\npartdw{n}$ and $\npart{n}{}$ (that is, a transport plan between $\nempdif{n}$ and $\nemp{n}$) with least maximum-matching cost $c_{\overline{k}}= \mathcal{W}^2_\infty\left(\nempdif{n},\nemp{n}\right)$. The algorithm terminates with a worst-case running time of $\mathcal{O}(N^2(N^2-k_0))$, but in practice this appears closer to $\mathcal{O}(N^2)$, since it is rarely the case that $\overline{k} >> k_0$.
\subsection{Monte Carlo Algorithm}
To calculate the expectation in Theorem \ref{thmparticle}, we use a Monte-Carlo sampling algorithm with 250 sample trajectories for each value of $\nu$. For the $m$th sample trajectory with diffusion coefficient $\nu$ and timestep $\Delta t$ the algorithm is as follows:
\begin{enumerate}
\item Generate random initial positions $\npart{0}{} = \npartdw{0} = \left\{X^i_{0}\right\}_{i=1}^N$ with law$\left(X^i_0\right) = \rho_0$ for each $i$.
\item Simulate $N \times \lfloor T/\Delta t\rfloor$ independent standard Gaussian random variables in $d$-dimensions, $N^d(0,1)$.
\item Advance $\npart{n}{}$ and $\npartdw{n}$ independently in time according to \eqref{eq:detpart} and \eqref{eq:stochpart}.
\item At times $\tau_k$ of interest, use \eqref{eqn:Wdist} to compute the distance between the particle systems.
\end{enumerate}
Let $\avW{\tau_k} \approx \EE\left[\mathcal{W}_\infty^2(\mu_{\tau_k}^{X,\nu},\mu_{\tau_k}^X)\right]$ be the sample average of the squared Wasserstein distances between the two particle systems at times $\tau_k$. We then compute the convergence rate $p$ such that $\avW{\tau_k} = \mathcal{O}(\nu^p)$ and use a Student's $t$-distribution to compute 95\% confidence intervals for $\avW{\tau_k}$  as in \cite{kloeden1995numerical}.

\subsection{Model Parameters}
For all simulations, the number of particles is fixed at $N=1000$, initial positions are drawn from a uniform square distribution, and zero external potential is used ($V=0$). We compute the convergence rate for times $\tau_k$ in the interval $[0,1]$. 

The domains we consider are the half-plane $\hp =  [0, \infty) \times \R$, and the disk $\disk = \{|x| \leq 0.2\} $. We note that although $\hp$ is unbounded, particle systems are naturally confined by the attractive forces, hence $\hp$ is essentially equivalent to a \textit{bounded} domain which is much larger than the support of the particle systems. This represents a case of self-organization near boundaries where there is abundant free space to escape the boundary, in contrast to $\disk$, which represents a situation where the support of the swarm is confined. The two cases lead to starkly different boundary interactions at large times, as seen by the snapshots at $T=100$ in Figures \ref{simdh} and \ref{simdc}.

To highlight the reduced regularity required of the interaction potential in the theorems above, we compare convergence results for attractive-repulsive interaction potentials of two regularities: 
\begin{equation*}
\hspace{-1.5cm}\klip(x) := \frac{1}{2}|x|^2 \hspace{0.3cm}+\begin{cases} 
\frac{1-2\pi\log(\epsilon)}{4\pi}-\frac{1}{4\pi\epsilon^2}|x|^2, &\quad |x| \in[0,\epsilon)\\
-\frac{1}{2\pi}\log |x|, & \quad |x|\in[\epsilon, \infty)
\end{cases} \hspace{0.2cm} \quad \implies \quad K_2 \in W^{2,\infty}(\R^2)
\end{equation*}
\begin{equation*}
\knlip(x) := \frac{2}{3}|x|^{3/2} + \begin{cases} 
\frac{1-2\pi\log(\epsilon)}{4\pi}-\frac{1}{4\pi\epsilon^2}|x|^2, & \quad |x|\in[0,\epsilon)\\
-\frac{1}{2\pi}\log|x|, &\quad |x|\in[\epsilon, \infty).
\end{cases} \quad \implies \quad K_{3/2}\in C^1(\R^2) \setminus W^{2,\infty}(\R^2).
\end{equation*}
Both $\klip$ and $\knlip$ are $\lambda$-convex and feature an attraction term along with a $C^1$ regularization of the repulsive Newtonian potential $\phi(x) = (-\Delta)^{-1}$ which introduces jump discontinuities in the second derivative at $r=\epsilon$ (where $\epsilon$ is fixed at $0.05$). Quadratic attraction puts $\klip$ in $W^{2,\infty}(\R^2)$, while the ${3/2}$-attraction term gives $\nabla \knlip$ a Lipschitz singularity at the origin. Potentials with non-Lipschitz gradients offer a generalization of the previous analyses for the zero-diffusion limit of aggregation-diffusion equations in bounded domains \cite{zhang2017continuity,fetecau2017swarming}. Specifically, $\klip$ has sufficient regularity for the analytical results in \cite{zhang2017continuity} to apply, as opposed to $\knlip$, which satisfies the relaxed conditions in our Assumption \ref{asum1}, but not the requirements in \cite{zhang2017continuity}. With this study we also draw attention to the improved $\mathcal{O}(\nu)$ convergence rate for such potentials.

\subsection{Results}
\begin{figure}
\begin{tabular}{cc}
        \includegraphics[trim={50 20 30 25},clip,width=0.35\textwidth]{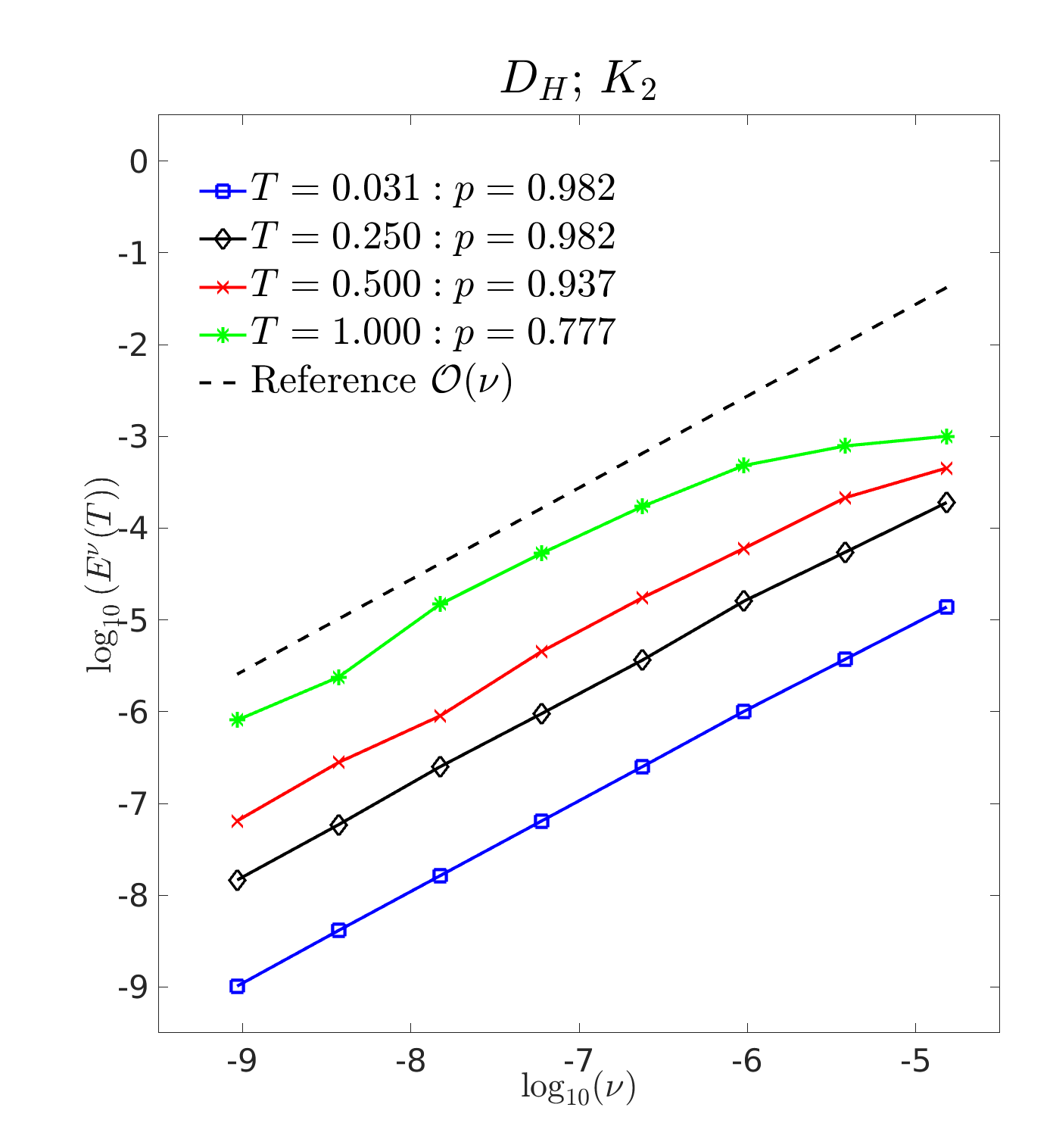} & \quad
        \includegraphics[trim={50 20 30 25},clip,width=0.35\textwidth]{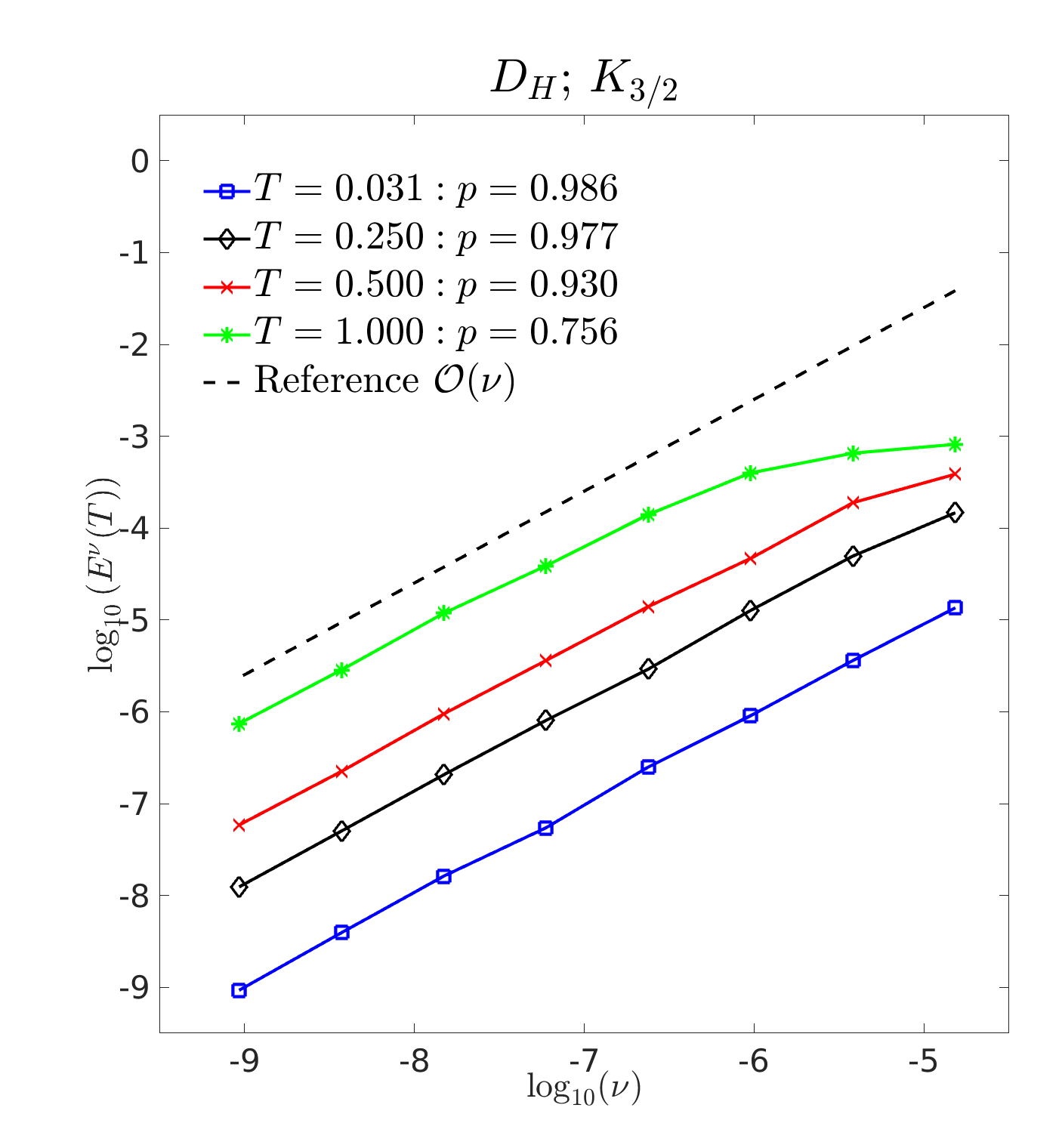}\\
        \includegraphics[trim={50 20 30 25},clip,width=0.35\textwidth]{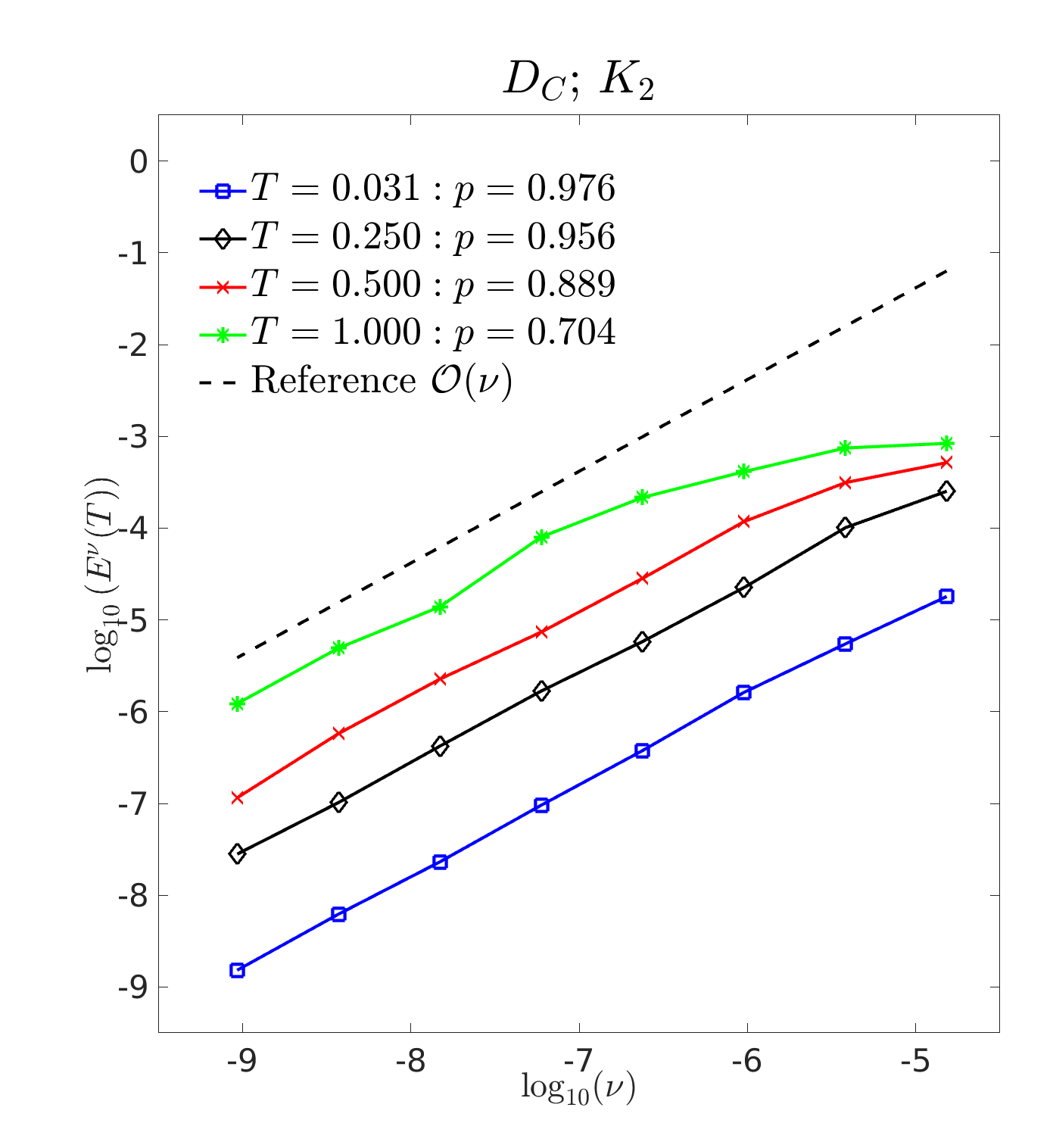}  & \quad 
        \includegraphics[trim={50 20 30 25},clip,width=0.35\textwidth]{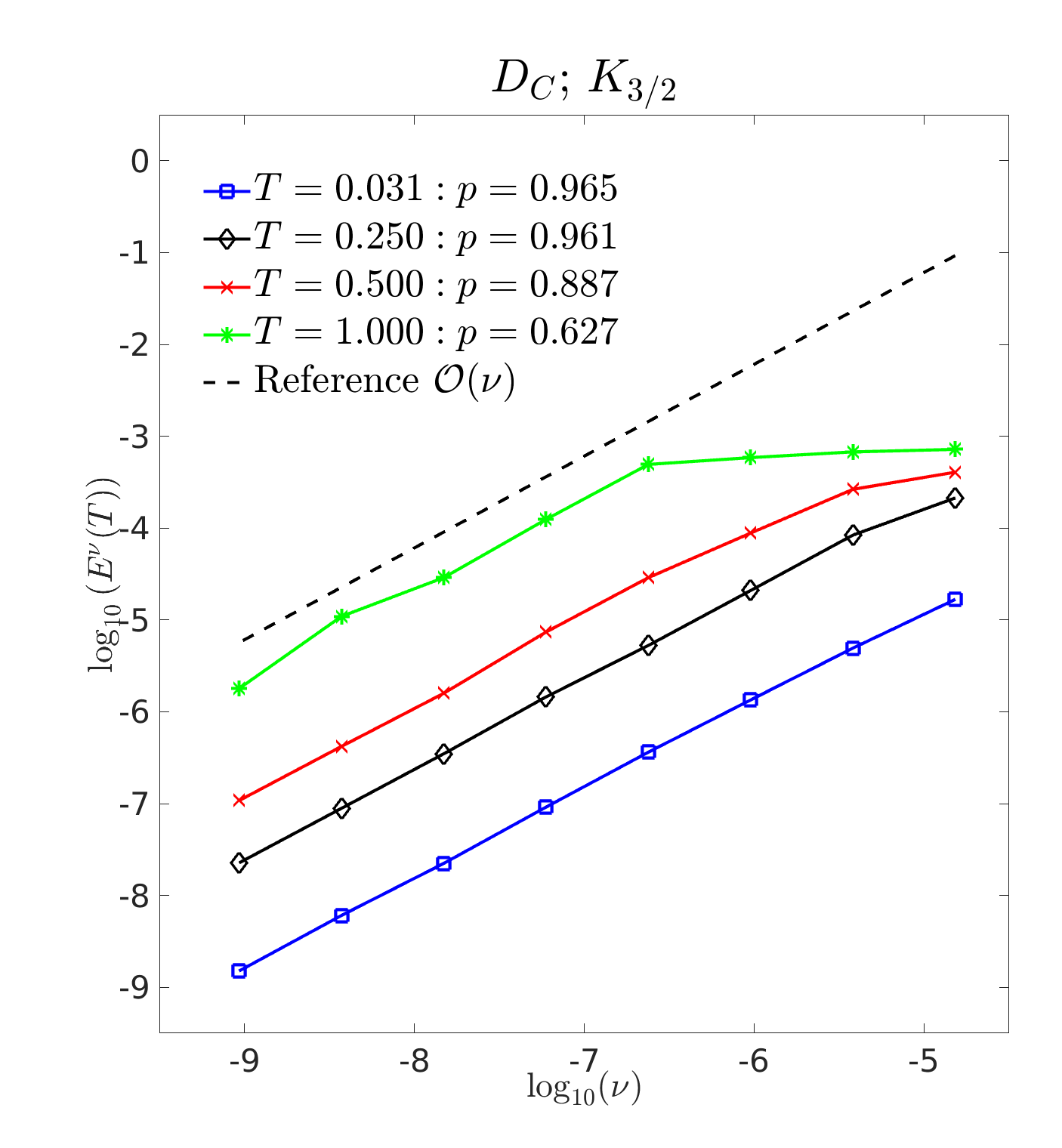}
\end{tabular}
\caption{(Top) Convergence results for $\hp = [0, \infty)\times \R$ with uniform initial distribution on $[0, 0.25]\times[-0.125, 0.125]$. (Bottom) Convergence results for $\disk = \{|x| \leq 0.2\} $ with uniform initial distribution on $[-0.05, 0.05]\times[0, 0.1]$. Left: $K_2$, right: $K_{3/2}$. Agreement with the theoretical rate of $\mathcal{O}(\nu)$ is very strong up to times near $T=1$, after which (for reasons discussed in the text) convergence at the expected rate is seen asymptotically, below a threshold value of $\nu$.}
\label{conv_all}
\end{figure}
\begin{figure}
\begin{tabular}{cc}
    \includegraphics[trim={10 8 30 5},clip,width=0.4\textwidth]{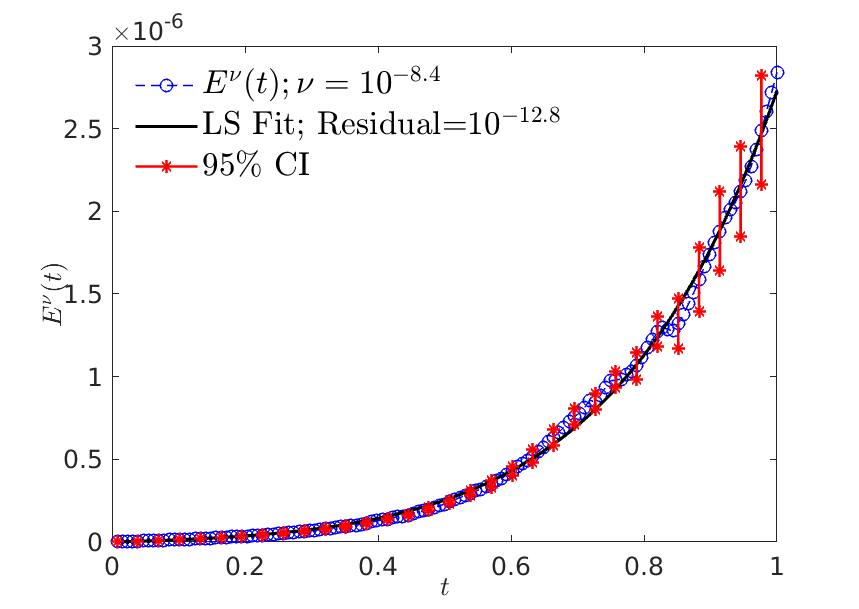} &
    \includegraphics[trim={10 8 30 5},clip,width=0.4\textwidth]{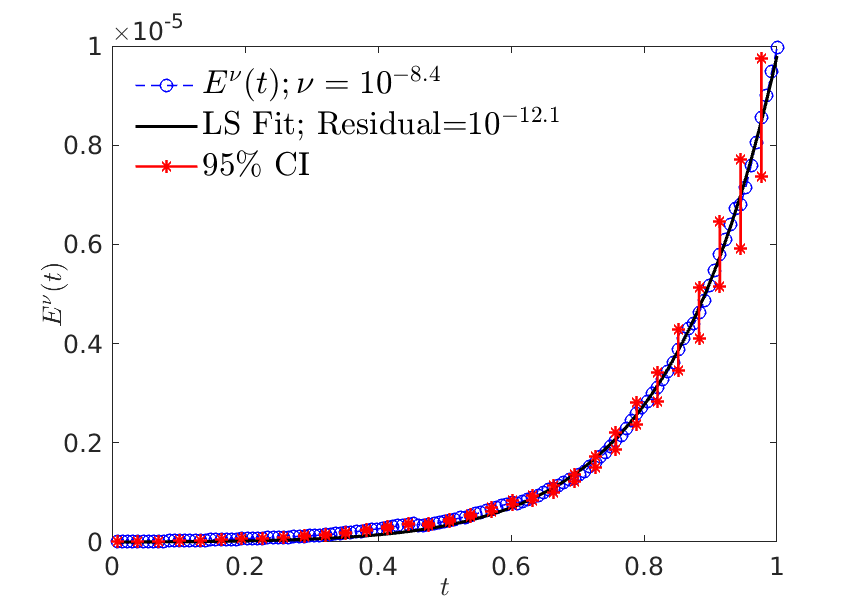}
\end{tabular}
    \caption{Evolution of $\avW{t}$ over time with $\nu = 2^{-28} \approx 10^{-8.4}$ and interaction potential $\knlip$ for the half-plane $\hp$ (left) and disk $\disk$ (right). As expected, variance of $\avW{t}$ grows significantly over time (as seen by the 95\% confidence intervals), yet we see excellent agreement using a nonlinear least-squares fit to a curve of the form $y(t) = at(1+bte^{bt})$, matching that of the bounding curve in Theorem \ref{thmparticle}.}
\label{exp_growth}
\end{figure}
\begin{figure}
\begin{tabular}{ccc}
        \includegraphics[trim={20 15 40 30},clip,width=0.32\textwidth]{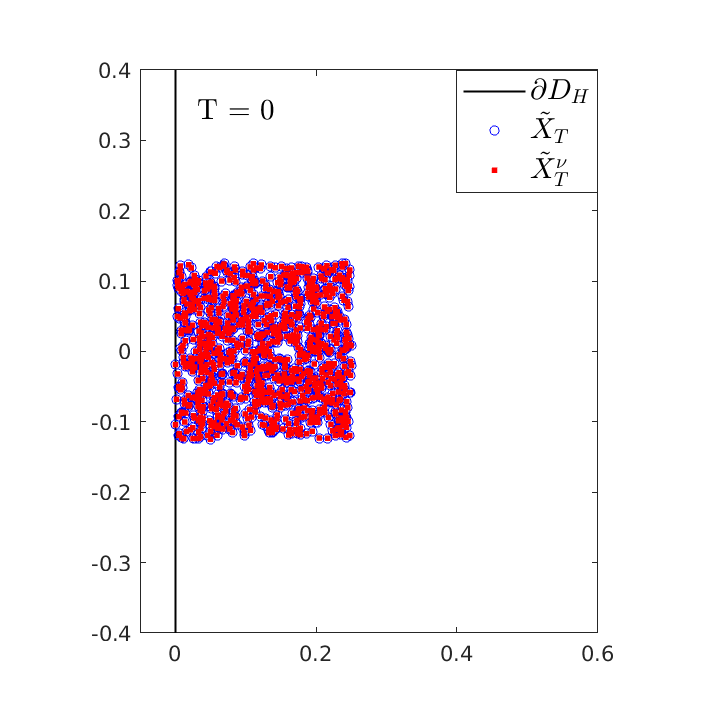} & 
        \includegraphics[trim={20 15 40 30},clip,width=0.32\textwidth]{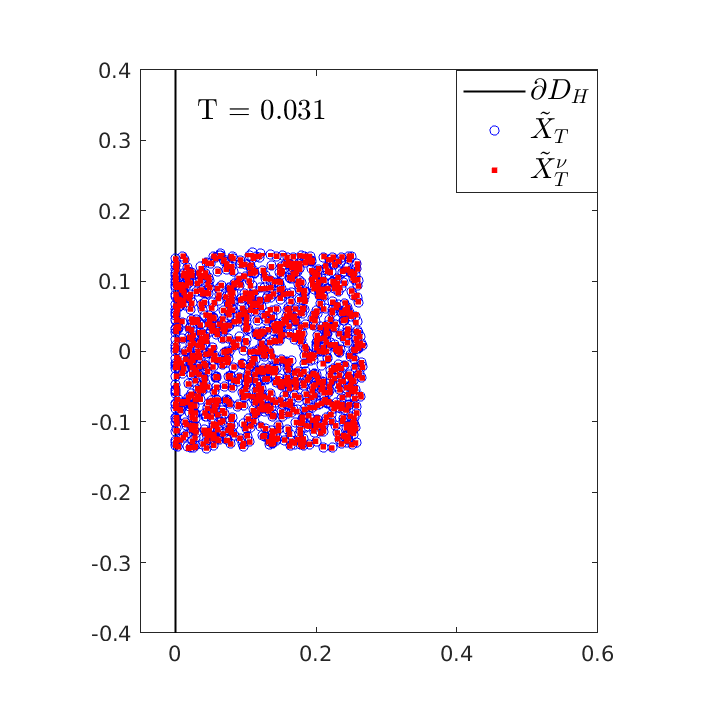} &
        \includegraphics[trim={20 15 40 30},clip,width=0.32\textwidth]{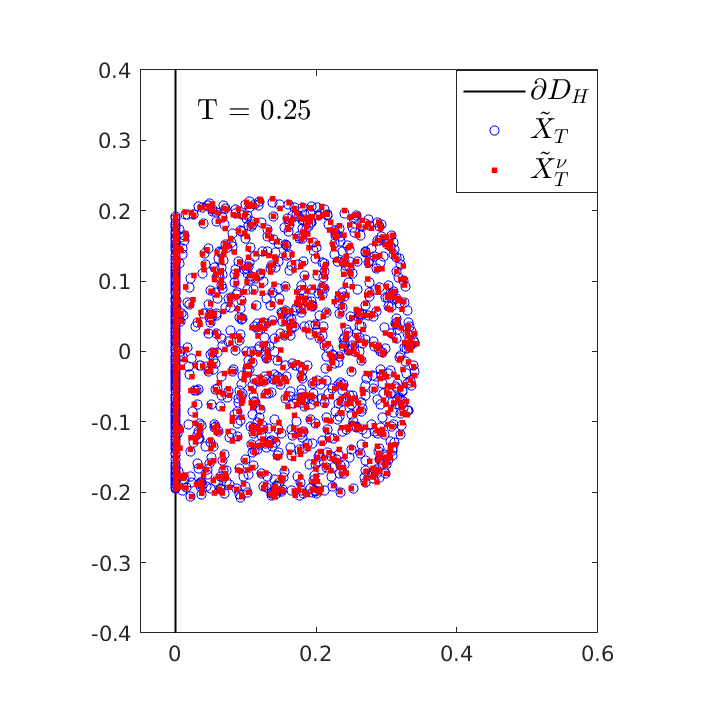} \\ 
        \includegraphics[trim={20 15 40 30},clip,width=0.32\textwidth]{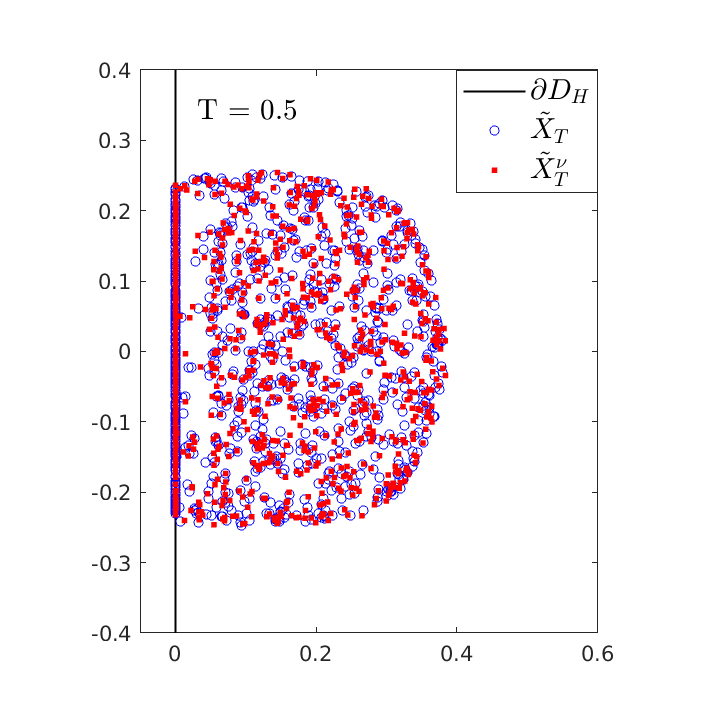} &
        \includegraphics[trim={20 15 40 30},clip,width=0.32\textwidth]{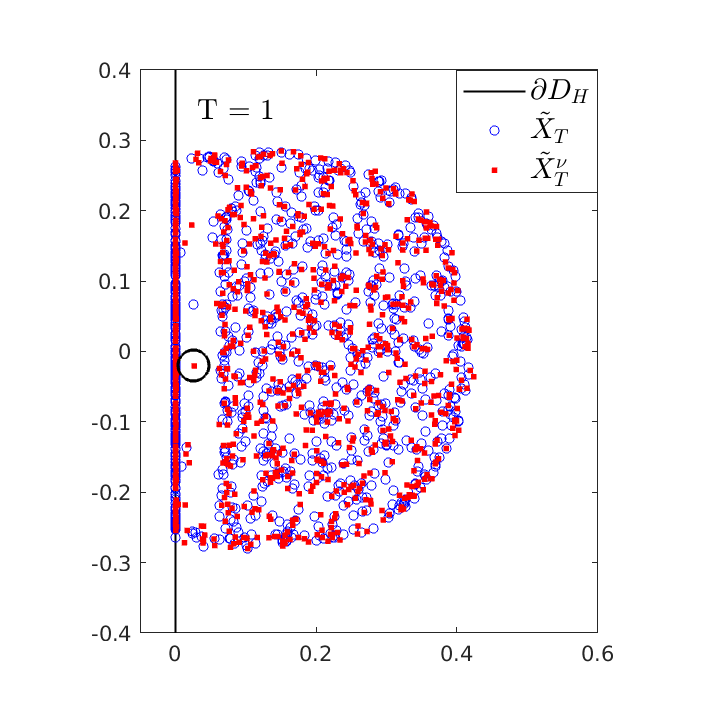} &
        \includegraphics[trim={20 15 40 30},clip,width=0.32\textwidth]{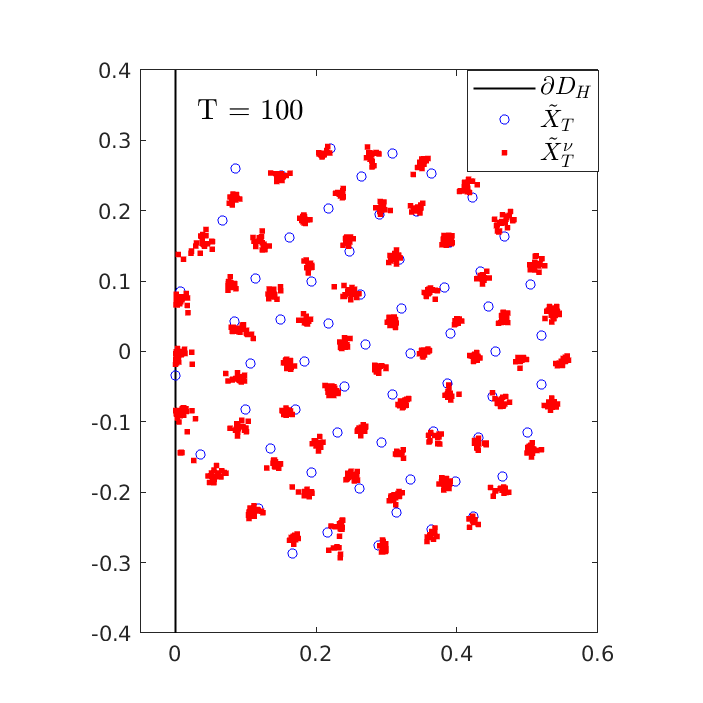}
\end{tabular}
\caption{Swarming in the half-plane $\hp$ with the interaction potential $\knlip$; here $\Delta t = 2^{-8}$ and $\nu = 2^{-16}$. Circles and filled-in squares represent non-diffusive and diffusive particles,  respectively. Particles rapidly expand until $T\approx 1$, after which attractive forces confine the swarm. At $T=1$ a representative particle is circled in the space between the boundary aggregation and the free swarm, illustrating the pair-separation effect mentioned in the text. By time $T=100$ the swarm has nearly escaped the boundary, forming a disk in free space, with diffusive particles forming concentrated clumps and non-diffusive particles forming $\delta$-aggregations.} 
\label{simdh}
\end{figure}
\begin{figure}
\begin{tabular}{ccc}
        \includegraphics[trim={10 15 20 30},clip,width=0.32\textwidth]{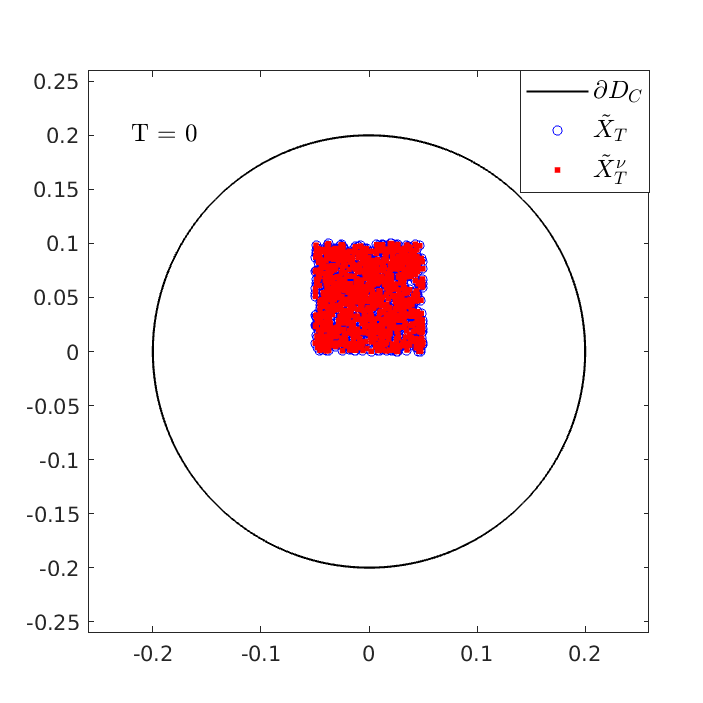} & 
        \includegraphics[trim={10 15 20 30},clip,width=0.32\textwidth]{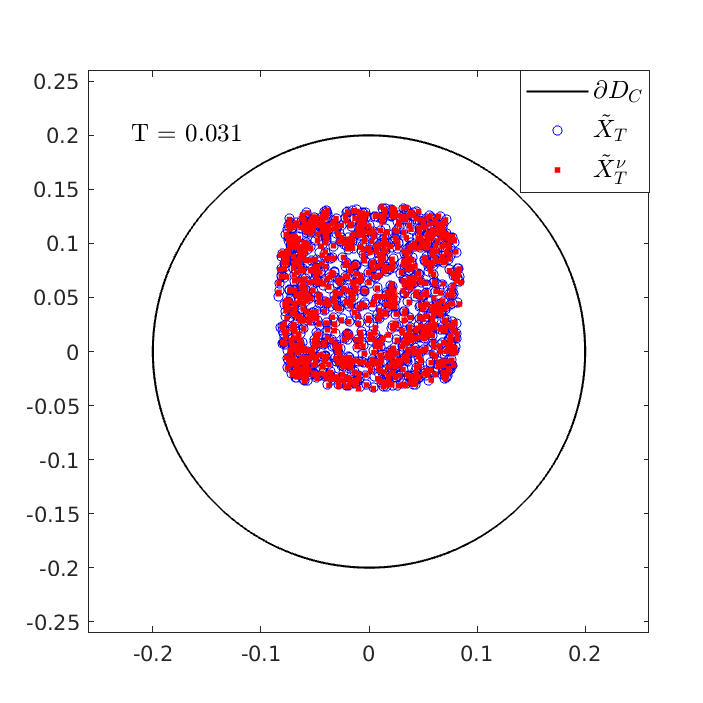} &
        \includegraphics[trim={10 15 20 30},clip,width=0.32\textwidth]{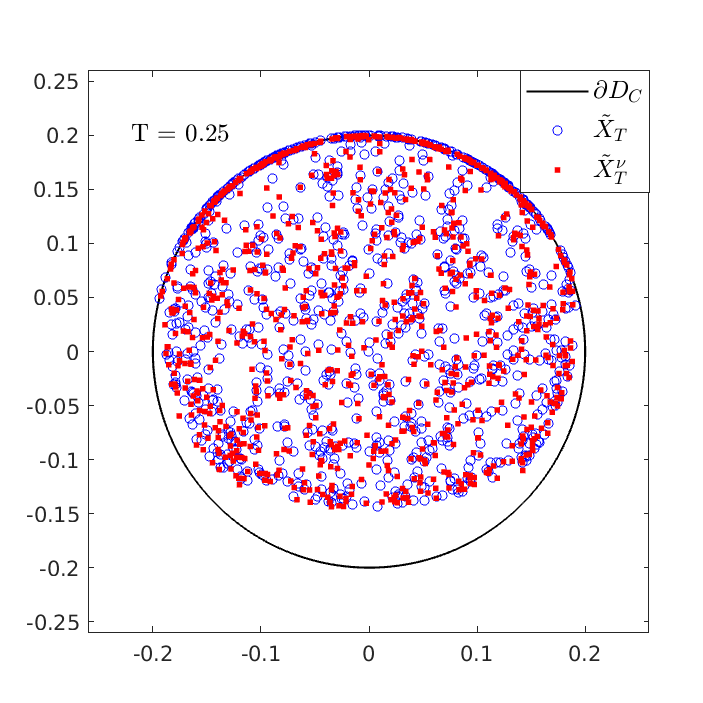} \\ 
        \includegraphics[trim={10 15 20 30},clip,width=0.32\textwidth]{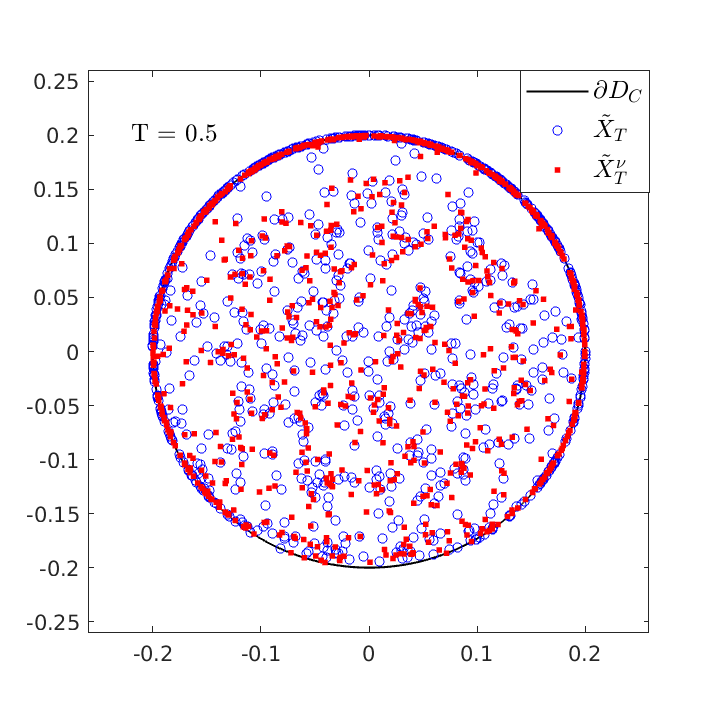} &
        \includegraphics[trim={10 15 20 30},clip,width=0.32\textwidth]{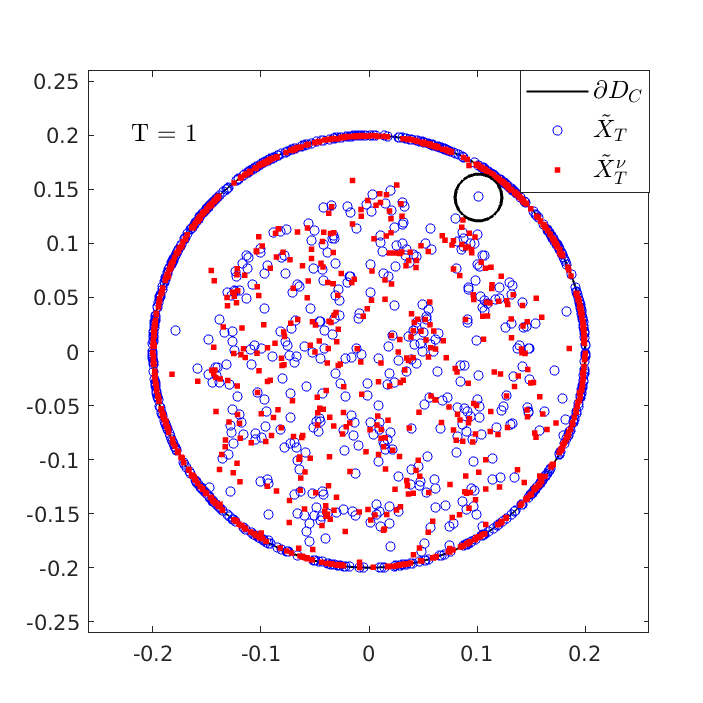} &
        \includegraphics[trim={10 15 20 30},clip,width=0.32\textwidth]{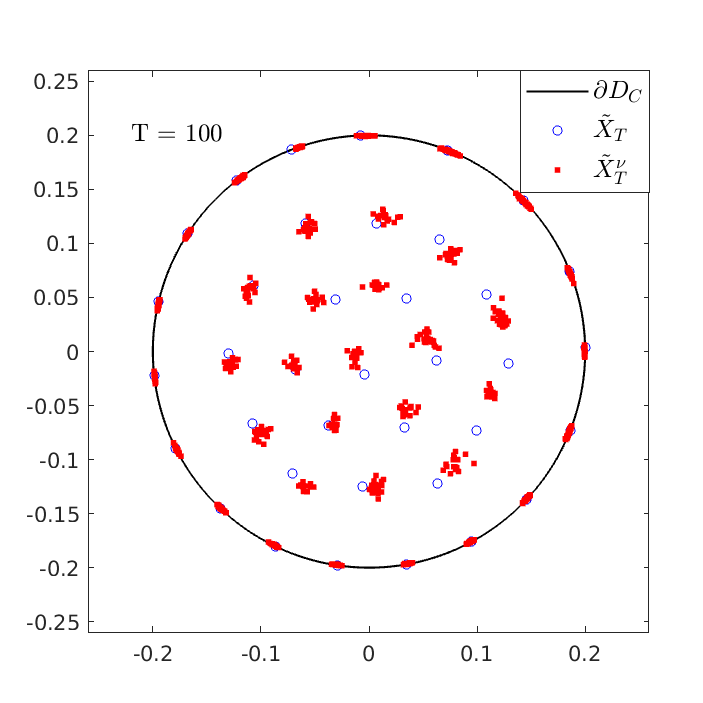}
\end{tabular}
\caption{Swarming in the disk $\disk$ with $K, \nu$ an $\Delta t$ matching that of Figure \ref{simdh}. By time $T\approx 1$, particles have formed a boundary aggregation and a free swarm as in domain $\hp$, with another representative particle (circled) caught between the boundary and free swarm. By time $T=100$ the swarm has formed a disk of particle aggregates centred in the domain along with periodic aggregations along the boundary.} 
\label{simdc}
\end{figure}

Using the Monte-Carlo algorithm above, we determine the convergence rate for each of the four $\{$domain, potential$\}$ combinations: $\{\hp, \klip\}$, $\{\hp,\knlip\}$, $\{\disk,\klip\}$ and $\{\disk,\knlip\}$ -- see Figure \ref{conv_all}. At times $T=0.031$, $T=0.25$ and $T=0.5$ the desired convergence of $\mathcal{O}(\nu)$ is clearly achieved, while by time $T=1$ the convergence behaviour becomes difficult to resolve numerically due to statistical errors (see in particular the growth of confidence intervals over time in Figure \ref{exp_growth}).

Along with convergence in $\nu$, we also examine the divergence in time of solutions on the interval $t\in[0,1]$. We observe that the divergence between empirical measures follows an exponential curve which matches the form of the bounding curve in Theorem \ref{thmparticle} to high accuracy, lending further numerical support to the theorem. Shown in Figure \ref{exp_growth}, the data $\left\{\tau_n,\avW{\tau_n}\right\}_{n=0}^L$ obtains a nonlinear least-squares residual of less than $10^{-12}$ when fit to a curve of the form $y(t) = at(1+bte^{bt})$. We include this result for the case of diffusion coefficient $\nu = 2^{-28} \approx 10^{-8.4}$ and potential $\knlip$, where the approximate exponential growth coefficient takes the value $b\approx 2.37$ on the half-plane and $b\approx 4.14$ in the disk, hence Theorem \ref{thmparticle} is satisfied as the analytical value of $b$ is two orders of magnitude larger, at $b = -2\lambda_K^- =\frac{1}{\pi\epsilon^2} \approx 127$.

We also observe some interesting phenomena when $\nu$ is larger than some threshold value which impacts the convergence rate and deserves some explanation. By time $T=1$ the swarm has separated into an aggregation on the boundary and a component in the interior of the domain which we will refer to as the ``free swarm". Particles occupying the space in between the boundary aggregation and the free swarm are either pinned to the boundary with high velocity or pulled gradually into the free swarm. This seems to be a generic effect of swarming near boundaries, occurring regardless of the choice of domain or interaction potential. With high probability in each simulation, for some $i$ the particle pair $\npartd{n}{i}$ and $\npart{n}{i}$ (sharing the same initial position) ends up separated, with one particle pulled onto the boundary and the other particle left drifting towards the free swarm (examples of this are circled in Figures \ref{simdh} and \ref{simdc} at time $T=1$). This pair-separation effect puts a lower bound on $\mathcal{W}^2_\infty\left(\nempdif{n},\nemp{n}\right)$ which disappears for small $\nu$. This seems to be the reason that $\mathcal{O}(\nu)$ convergence is not visible in the numerics for $\nu > 10^{-6}$ in the half-plane, and $\nu > 10^{-7}$ on the disk. 

The strong force with which particles are pinned to the boundary along with the resulting energy barrier between the boundary aggregation and the free swarm are some of the many fascinating differences between swarms in free space and those in domains with boundaries. As shown in \cite{fetecau2017swarm}, the plain aggregation model \eqref{PDE} in domains with boundaries tends to evolve into unstable equilibria. While the numerical studies in the present paper are exclusively designed to support Theorem \ref{thmparticle}, we plan for future work a thorough investigation of the long-time behaviour of solutions to model \eqref{PDEnu}. Of particular interest is to investigate whether/how solutions of the diffusive model (with small diffusion) bypass the unstable equilibria of the plain aggregation equation. Such regularizing effect was recently demonstrated in \cite{fetecau2017swarming} using nonlinear diffusion; however, due to the computational complexity of the PDE integrators at small diffusion values, the numerical studies there were limited to one dimension. The main advantage of the present approach is the discrete (stochastic) formulation, which potentially can enable a numerical study of the zero diffusion limit for long times in higher dimensions.

\subsection{Remarks}
We now make some remarks about the numerical method, and specifically about the symmetrized Euler scheme: 
\smallskip

 \textit{i) Weak} $\mathcal{O}(\Delta t)$: To the best of our knowledge, a numerical method for treating \textit{reflected} SDEs with a weak convergence rate higher than $\mathcal{O}(\Delta t)$ is currently lacking. It is shown in \cite{bossy2004symmetrized} that for drift coefficients
 in $C^4_b\left(D\right)$ and test functions $g \in C^5_b\left(D\right)$, the symmetrized Euler scheme exhibits the weak convergence
\begin{equation}\label{eq:weakerr_num}
e_g^\text{weak}(\Delta t) := \left\vert\mathbb{E}\left[g\left(\npartdw{n}\right)\right]- \mathbb{E}\left[g\Big(X^\nu_{\tau_n}\Big)\right]\right\vert = \tilde{C}\Delta t.
\end{equation}
To accurately compute the $\mathcal{W}^2_\infty$-distance in \eqref{eqn:Wdist}, we require weak convergence for test functions of the form $g(x) = |x-a|^2$. Such test functions are smooth, yet due to the reduced regularity of the interaction potentials in this study, we fail to meet the smoothness requirements on the drift coefficient used in \cite{bossy2004symmetrized} to prove weak $\mathcal{O}(\Delta t)$ convergence. It appears that rigorous analysis of the convergence for numerical solutions to reflected SDEs with non-Lipschitz drift is also lacking, let alone analysis for schemes applied to interacting particle systems. Hence, to justify our method, we provide numerical support instead; see Figure \ref{conv_sde} showing $e_g^\text{weak}(\Delta t) = \mathcal{O}(\Delta t)$ for particle systems using the parameters of this study.
\begin{figure}
\centering
	\includegraphics[trim={0 0 0 40},clip,width=1\textwidth]{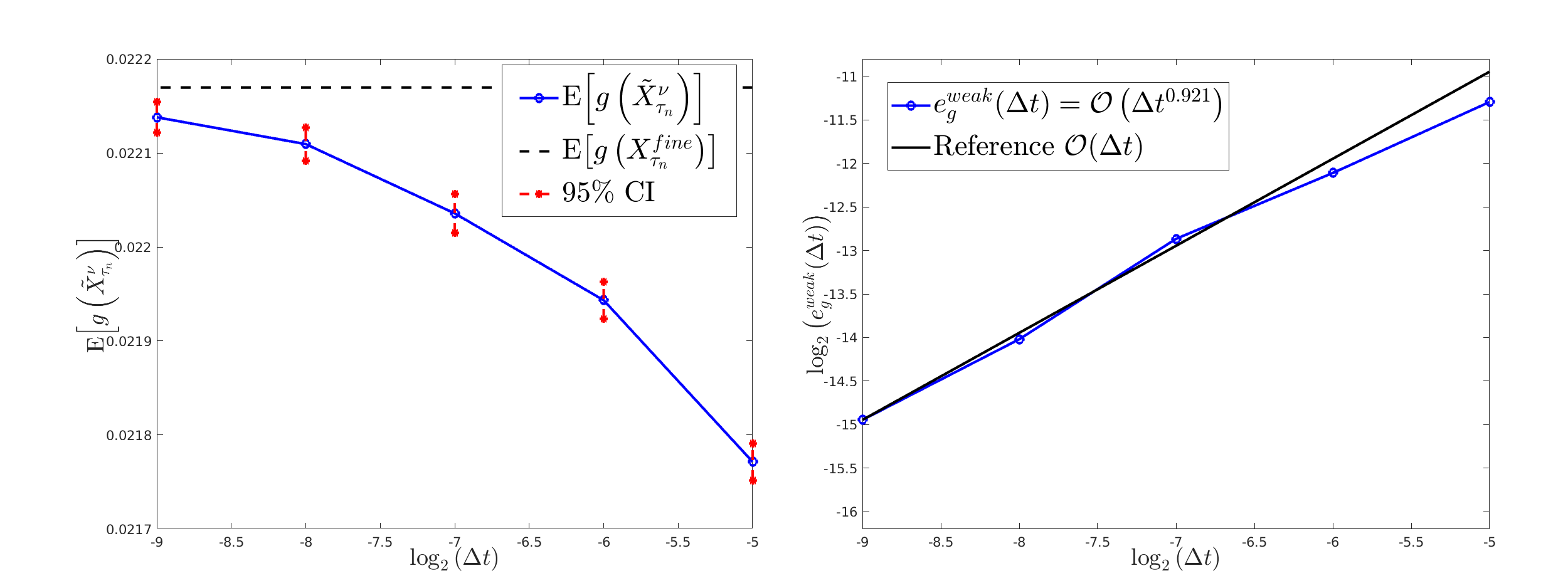}
	\caption{Weak convergence of the Euler scheme with symmetric reflection \eqref{eq:stochpart}, with diffusion coefficient $\nu = 0.01$, $N=5$ particles,  domain $D_C$ and interaction potential $\knlip$. We use the test function $g(\npartdw{n}) = \frac{1}{N}\sum_{i=1}^N\left\vert \npartd{n}{i}\right\vert^2$ which is motivated by the computation of \eqref{eqn:Wdist} as discussed in the text. In lieu of an exact solution we compare at time $\tau_n = 0.25$ to a ``fine-grid" solution $X^{fine}_{\tau_n}$ computed with timestep $\Delta t = 2^{-15}$. To approximate expectations we use $2\times 10^5$ samples of Brownian motion, each simulation sharing the same initial data. Left: semi-log plot showing convergence of $\mathbb{E}\left[g\left(\npartdw{n}\right)\right]$ to $\mathbb{E}\left[g\left(X^{fine}_{\tau_n}\right)\right]$ as well as 95\% confidence intervals. Right: log-log plot showing convergence rate at approximately $\mathcal{O}(\Delta t)$.}
\label{conv_sde}
\end{figure}
\smallskip

 \textit{ii) Refinement of} $\Delta t$:
To observe the analytical $\mathcal{O}(\nu)$ convergence rate numerically, it is necessary to suitably reduce the time-discretization error. This presents a serious bottleneck in computation, hence, instead of fixing a prohibitively small $\Delta t$ for every value of $\nu$, we successively refine $\Delta t$ using the heuristic scaling $\Delta t \sim \sqrt{\nu}$, which is based on the following argument. Computing \eqref{eqn:Wdist} relies on accurate calculation of the expectations $\mathbb{E}\left[\left\vert\npartd{n}{i}-\npart{n}{j}\right\vert^2\right]$. Fix $\tau_n$ and treat $\nu$ as an explicit parameter in the weak error \eqref{eq:weakerr_num}. For a given particle pair $(i,j)$, let $g_j(x) = \bigl \vert x-\npart{n}{j}\bigr\vert^2$ and define,
\begin{equation}\label{eq:weakerr_num1}
 e^{i,j}(\Delta t; \nu) := \left\vert\mathbb{E}\left[g_j\left(\npartd{n}{i}\right)\right]- \mathbb{E}\left[g_j\left(X^{\nu,i}_{\tau_n}\right)\right]\right\vert,
\end{equation}
that is, the weak error in approximating the particle $X^{\nu,i}_t$ at time $t=\tau_n$ using the test function $g_j$. We then observe that
\begin{align*}
\mathbb{E}\left[\left\vert\npartd{n}{i}-\npart{n}{j}\right\vert^2\right] &= \mathbb{E}\left[g_j\left(\npartd{n}{i}\right)\right] \\
& \leq e^{i,j}(\Delta t; \nu)+\mathbb{E}\left[g_j\left(X^{\nu,i}_{\tau_n}\right)\right]\\
&\leq e^{i,j}(\Delta t; \nu)+\mathbb{E}\left[\left(\left\vert X^{\nu,i}_{\tau_n}-X^j_{\tau_n}\right\vert+\left\vert X^j_{\tau_n}-\npart{n}{j}\right\vert\right)^2\right]\\
&= e^{i,j}(\Delta t; \nu)+\mathbb{E}\left[\left\vert X^{\nu,i}_{\tau_n}-X^j_{\tau_n}\right\vert^2\right]+2\left\vert X^j_{\tau_n}-\npart{n}{j}\right\vert\mathbb{E}\left[\left\vert X^{\nu,i}_{\tau_n}-X^j_{\tau_n}\right\vert\right]
+\left\vert X^j_{\tau_n}-\npart{n}{j}\right\vert^2\\
&\leq e^{i,j}(\Delta t; \nu) + \tilde{C}\left(\nu+\sqrt{\nu}\Delta t+\Delta t^2\right),
\end{align*}
for $\tilde{C}$ not depending on $\nu$ or $\Delta t$, where we have used the triangle inequality, the deterministic forward Euler error, $\left\vert X^j_{\tau_n}-\npart{n}{j}\right\vert = \mathcal{O}(\Delta t)$ and the result from \eqref{thmparticle}. Taking supremum over $(i,j)$ and using $\Delta t = \sqrt{\nu}$, we arrive at
\begin{equation}
\mathbb{E}\left[\mathcal{W}^2_\infty\left(\nempdif{n},\nemp{n}\right)\right]\leq \underbrace{\sup_{\left\{(i,j)\right\}}\left\{e^{i,j}(\sqrt{\nu}; \nu)\right\}}_{E1} +\tilde{C}\nu,
\end{equation}
justifying the scaling $\Delta t = \sqrt{\nu}$ as long as $E1$ is sufficiently small. Numerical evidence (see Figure \ref{conv_sde}) suggests that $e^{i,j}( \Delta t;\nu) = \tilde{C}(\nu)\Delta t$, and hence that $E1 = \tilde{C}(\nu)\sqrt{\nu}$, although the form of $\tilde{C}(\nu)$ is unknown. Judging by the convergence results in Figure \ref{conv_all}, however, it appears that $\tilde{C}(\nu) = \mathcal{O}(\sqrt{\nu})$, or else such convergence would not be attainable asymptotically. So we adopt the refinement $\Delta t = \sqrt{\nu}$ preemptively for now to cut computation costs, and leave rigorous justification of the scaling $\Delta t \sim \sqrt{\nu}$ for future work. 

\bibliographystyle{amsxport}
\bibliography{diffusion}

\begin{bibdiv}
\begin{biblist}

\bib{ambrosio2008gradient}{book}{
      author={Ambrosio, Luigi},
      author={Gigli, Nicola},
      author={Savar{\'e}, Giuseppe},
       title={Gradient flows: in metric spaces and in the space of probability
  measures},
   publisher={Springer Science \& Business Media},
        date={2008},
}

\bib{bedrossian2011local}{article}{
      author={Bedrossian, Jacob},
      author={Rodr{\'\i}guez, Nancy},
      author={Bertozzi, Andrea~L},
       title={Local and global well-posedness for aggregation equations and
  {Patlak-Keller-Segel} models with degenerate diffusion},
        date={2011},
     journal={Nonlinearity},
      volume={24},
      number={6},
       pages={1683},
}

\bib{bertozzi2009blow}{article}{
      author={Bertozzi, Andrea~L},
      author={Carrillo, Jos{\'e}~A},
      author={Laurent, Thomas},
       title={Blow-up in multidimensional aggregation equations with mildly
  singular interaction kernels},
        date={2009},
     journal={Nonlinearity},
      volume={22},
      number={3},
       pages={683},
}

\bib{bertozzi2007finite}{article}{
      author={Bertozzi, Andrea~L},
      author={Laurent, Thomas},
       title={Finite-time blow-up of solutions of an aggregation equation in
  $\mathbb{R}^n$},
        date={2007},
     journal={Communications in Mathematical Physics},
      volume={274},
      number={3},
       pages={717\ndash 735},
}

\bib{bertozzi2012aggregation}{article}{
      author={Bertozzi, Andrea~L},
      author={Laurent, Thomas},
      author={L{\'e}ger, Flavien},
       title={Aggregation and spreading via the {Newtonian} potential: the
  dynamics of patch solutions},
        date={2012},
     journal={Mathematical Models and Methods in Applied Sciences},
      volume={22},
      number={supp01},
       pages={1140005},
}

\bib{bertozzi2009existence}{article}{
      author={Bertozzi, Andrea~L},
      author={Slepcev, Dejan},
       title={Existence and uniqueness of solutions to an aggregation equation
  with degenerate diffusion},
        date={2009},
     journal={Communications on Pure and Applied Analysis},
      volume={9},
      number={6},
       pages={1617},
}

\bib{bossy2004symmetrized}{article}{
      author={Bossy, Mireille},
      author={Gobet, Emmanuel},
      author={Talay, Denis},
       title={A symmetrized {E}uler scheme for an efficient approximation of
  reflected diffusions},
        date={2004},
     journal={Journal of Applied Probability},
      volume={41},
      number={3},
       pages={877\ndash 889},
}

\bib{carrillo2017blob}{article}{
      author={Carrillo, Jos{\'e}~A},
      author={Craig, Katy},
      author={Patacchini, Francesco~S},
       title={A blob method for diffusion},
        date={2017},
     journal={arXiv preprint arXiv:1709.09195},
}

\bib{carrillo2011global}{article}{
      author={Carrillo, Jos{\'e}~A},
      author={DiFrancesco, Marco},
      author={Figalli, Alessio},
      author={Laurent, Thomas},
      author={Slep{\v{c}}ev, Dejan},
       title={Global-in-time weak measure solutions and finite-time aggregation
  for nonlocal interaction equations},
        date={2011},
     journal={Duke Mathematical Journal},
      volume={156},
      number={2},
       pages={229\ndash 271},
}

\bib{carrillo2003kinetic}{article}{
      author={Carrillo, Jos{\'e}~A},
      author={McCann, Robert~J},
      author={Villani, C{\'e}dric},
      author={others},
       title={Kinetic equilibration rates for granular media and related
  equations: entropy dissipation and mass transportation estimates},
        date={2003},
     journal={Revista Matematica Iberoamericana},
      volume={19},
      number={3},
       pages={971\ndash 1018},
}

\bib{carrillo2014nonlocal}{article}{
      author={Carrillo, Jos{\'e}~A},
      author={Slep{\v{c}}ev, Dejan},
      author={Wu, Lijiang},
       title={Nonlocal-interaction equations on uniformly prox-regular sets},
        date={2014},
     journal={Discrete and Continuous Dynamical Systems-Series A},
      volume={36},
      number={3},
       pages={1209\ndash 1247},
}

\bib{chen2018modeling}{article}{
      author={Chen, Li},
      author={G{\"o}ttlich, Simone},
      author={Knapp, Stephan},
       title={Modeling of a diffusion with aggregation: rigorous derivation and
  numerical simulation},
        date={2018},
     journal={arXiv preprint arXiv:1801.09074},
}

\bib{chertock2017practical}{incollection}{
      author={Chertock, Alina},
       title={A practical guide to deterministic particle methods},
        date={2017},
   booktitle={Handbook of numerical analysis},
      volume={18},
   publisher={Elsevier},
       pages={177\ndash 202},
}

\bib{choi2016propagation}{article}{
      author={Choi, Young-Pil},
      author={Salem, Samir},
       title={Propagation of chaos for aggregation equations with no-flux
  boundary conditions and sharp sensing zones},
        date={2018},
     journal={Mathematical Models and Methods in Applied Sciences},
      volume={28},
      number={02},
       pages={223\ndash 258},
}

\bib{craig2014blob}{article}{
      author={Craig, Katy},
      author={Bertozzi, Andrea},
       title={A blob method for the aggregation equation},
        date={2016},
     journal={Mathematics of Computation},
      volume={85},
       pages={1681\ndash 1717},
}

\bib{edmond2006bv}{article}{
      author={Edmond, Jean~Fenel},
      author={Thibault, Lionel},
       title={{BV} solutions of nonconvex sweeping process differential
  inclusion with perturbation},
        date={2006},
     journal={Journal of Differential Equations},
      volume={226},
      number={1},
       pages={135\ndash 179},
}

\bib{evers2016metastable}{article}{
      author={Evers, Joep~HM},
      author={Kolokolnikov, Theodore},
       title={Metastable states for an aggregation model with noise},
        date={2016},
     journal={SIAM Journal on Applied Dynamical Systems},
      volume={15},
      number={4},
       pages={2213\ndash 2226},
}

\bib{fetecau2018propagation}{article}{
      author={Fetecau, Razvan~C},
      author={Huang, Hui},
      author={Sun, Weiran},
       title={Propagation of chaos for the {Keller-Segel} equation over bounded
  domains},
        date={2018},
     journal={Journal of Differential Equations},
        note={in press, DOI 10.1016/j.jde.2018.08.024},
}

\bib{fetecau2017swarm}{article}{
      author={Fetecau, Razvan~C},
      author={Kovacic, Mitchell},
       title={Swarm equilibria in domains with boundaries},
        date={2017},
     journal={SIAM Journal on Applied Dynamical Systems},
      volume={16},
      number={3},
       pages={1260\ndash 1308},
}

\bib{fetecau2017swarming}{article}{
      author={Fetecau, Razvan~C},
      author={Kovacic, Mitchell},
      author={Topaloglu, Ihsan},
       title={Swarming in domains with boundaries: approximation and
  regularization by nonlinear diffusion},
        date={2017},
     journal={arXiv preprint arXiv:1711.03622},
}

\bib{filippov2013differential}{book}{
      author={Filippov, Aleksej~Fedorovi{\v{c}}},
       title={Differential equations with discontinuous righthand sides:
  control systems},
   publisher={Springer Science \& Business Media},
        date={2013},
      volume={18},
}

\bib{fournier2015rate}{article}{
      author={Fournier, Nicolas},
      author={Guillin, Arnaud},
       title={On the rate of convergence in {Wasserstein} distance of the
  empirical measure},
        date={2015},
     journal={Probability Theory and Related Fields},
      volume={162},
      number={3-4},
       pages={707\ndash 738},
}

\bib{fournier2014propagation}{article}{
      author={Fournier, Nicolas},
      author={Hauray, Maxime},
      author={Mischler, St{\'e}phane},
       title={Propagation of chaos for the 2d viscous vortex model},
        date={2014},
     journal={Journal of the European Mathematical Society},
      volume={16},
      number={7},
       pages={1423\ndash 1466},
}

\bib{fournier2015stochastic}{article}{
      author={Fournier, Nicolas},
      author={Jourdain, Benjamin},
       title={Stochastic particle approximation of the {Keller-Segel} equation
  and two-dimensional generalization of {Bessel} processes},
        date={2015},
     journal={arXiv preprint arXiv:1507.01087},
}

\bib{garcia2017}{article}{
      author={Garc{\'\i}a-Ca{\~n}izares, Ana},
      author={Pickl, Peter},
       title={Microscopic derivation of the {Keller-Segel} equation in the
  sub-critical regime},
        date={2017},
     journal={arXiv preprint arXiv:1703.04376},
}

\bib{gazi2003stability}{article}{
      author={Gazi, Veysel},
      author={Passino, Kevin~M},
       title={Stability analysis of swarms},
        date={2003},
     journal={IEEE Transactions on Automatic Control},
      volume={48},
      number={4},
       pages={692\ndash 697},
}

\bib{gihman1979stochastic}{book}{
      author={Gihman, Iosif~I},
      author={Skorohod, Anatolij~V},
       title={Stochastic differential equations},
   publisher={Springer-Verlag Berlin Heidelberg, New York},
        date={1972},
}

\bib{holm2006formation}{article}{
      author={Holm, Darryl~D},
      author={Putkaradze, Vakhtang},
       title={Formation of clumps and patches in self-aggregation of
  finite-size particles},
        date={2006},
     journal={Physica D: Nonlinear Phenomena},
      volume={220},
      number={2},
       pages={183\ndash 196},
}

\bib{HH2}{article}{
      author={Huang, Hui},
      author={Liu, Jian-Guo},
       title={Discrete-in-time random particle blob method for the
  {Keller-Segel} equation and convergence analysis},
        date={2017},
     journal={Communication in Mathematical Sciences},
      volume={15},
      number={7},
       pages={1821\ndash 1842},
}

\bib{HH1}{article}{
      author={Huang, Hui},
      author={Liu, Jian-Guo},
       title={Error estimate of a random particle blob method for the
  {Keller-Segel} equation},
        date={2017},
     journal={Mathematics of Computation},
      volume={86},
      number={308},
       pages={2719\ndash 2744},
}

\bib{huilearning}{article}{
      author={Huang, Hui},
      author={Liu, Jian-Guo},
      author={Lu, Jianfeng},
       title={Learning interacting particle systems: diffusion parameter
  estimation for aggregation equations},
        date={2018},
     journal={arXiv preprint arXiv:1802.02267},
}

\bib{jabin2017mean}{incollection}{
      author={Jabin, Pierre-Emmanuel},
      author={Wang, Zhenfu},
       title={Mean field limit for stochastic particle systems},
        date={2017},
   booktitle={Active particles, volume 1},
   publisher={Springer},
       pages={379\ndash 402},
}

\bib{kloeden1995numerical}{book}{
      author={Kloeden, Peter},
      author={Platen, Eckhard},
       title={Numerical solution of stochastic differential equations},
   publisher={Springer},
        date={1995},
}

\bib{lions1984stochastic}{article}{
      author={Lions, Pierre-Louis},
      author={Sznitman, Alain-Sol},
       title={Stochastic differential equations with reflecting boundary
  conditions},
        date={1984},
     journal={Communications on Pure and Applied Mathematics},
      volume={37},
      number={4},
       pages={511\ndash 537},
}

\bib{brent}{book}{
      author={\O{}ksendal, Bernt},
       title={Stochastic differential equations: An introduction with
  applications, sixth edition},
   publisher={Springer, Berlin},
        date={2003},
}

\bib{saisho1987stochastic}{article}{
      author={Saisho, Yasumasa},
       title={Stochastic differential equations for multi-dimensional domain
  with reflecting boundary},
        date={1987},
     journal={Probability Theory and Related Fields},
      volume={74},
      number={3},
       pages={455\ndash 477},
}

\bib{skorokhod1961stochastic}{article}{
      author={Skorokhod, Anatoliy~V},
       title={Stochastic equations for diffusion processes in a bounded
  region},
        date={1961},
     journal={Theory of Probability \& Its Applications},
      volume={6},
      number={3},
       pages={264\ndash 274},
}

\bib{skorokhod1962stochastic}{article}{
      author={Skorokhod, Anatoliy~V},
       title={Stochastic equations for diffusion processes in a bounded region.
  {II}},
        date={1962},
     journal={Theory of Probability \& Its Applications},
      volume={7},
      number={1},
       pages={3\ndash 23},
}

\bib{sznitman1984nonlinear}{article}{
      author={Sznitman, Alain-Sol},
       title={Nonlinear reflecting diffusion process, and the propagation of
  chaos and fluctuations associated},
        date={1984},
     journal={Journal of Functional Analysis},
      volume={56},
      number={3},
       pages={311\ndash 336},
}

\bib{tanaka1979stochastic}{article}{
      author={Tanaka, Hiroshi},
       title={Stochastic differential equations with reflecting},
        date={1979},
     journal={Stochastic Processes: Selected Papers of Hiroshi Tanaka},
      volume={9},
       pages={157},
}

\bib{villani2008optimal}{book}{
      author={Villani, C{\'e}dric},
       title={Optimal transport: old and new},
   publisher={Springer Science \& Business Media},
        date={2008},
      volume={338},
}

\bib{weinan1994dynamics}{article}{
      author={Weinan, E},
       title={Dynamics of vortex liquids in {Ginzburg-Landau} theories with
  applications to superconductivity},
        date={1994},
     journal={Physical Review B},
      volume={50},
      number={2},
       pages={1126},
}

\bib{wu2015nonlocal}{article}{
      author={Wu, Lijiang},
      author={Slep{\v{c}}ev, Dejan},
       title={Nonlocal interaction equations in environments with
  heterogeneities and boundaries},
        date={2015},
     journal={Communications in Partial Differential Equations},
      volume={40},
      number={7},
       pages={1241\ndash 1281},
}

\bib{zhang2017continuity}{article}{
      author={Zhang, Yuming},
       title={On continuity equations in space-time domains},
        date={2018},
     journal={Discrete and Continuous Dynamical Systems-Series A},
      volume={38},
      number={10},
       pages={4837�\ndash 4873},
}

\end{biblist}
\end{bibdiv}


\end{document}